\documentclass[12pt]{amsart}

\usepackage{mathptmx}
\usepackage{amssymb}
\usepackage{eucal}
\usepackage{amsxtra}
\usepackage{verbatim}
\usepackage{enumerate}
\usepackage{enumitem}
\usepackage[english]{babel}
\usepackage{bm}
\usepackage{newtxmath}
\usepackage{setspace}
\usepackage[x11names]{xcolor}

\usepackage[T1]{fontenc}

\usepackage{graphicx}
\usepackage{mathrsfs}
\usepackage{geometry}
\usepackage{amscd,latexsym,amsthm,amsfonts,amssymb,amsmath,amsxtra}
 \usepackage[colorlinks, linkcolor=DodgerBlue3, citecolor=Red2]{hyperref}
\usepackage{enumerate}
\usepackage[all]{xy}
\providecommand{\U}[1]{\protect\rule{.1in}{.1in}}

\RequirePackage{amsmath}
\RequirePackage{amssymb}

\usepackage{comment}

 \newtheorem{thm}{Theorem}[section]

 \newtheorem{cor}[thm]{Corollary}
 \newtheorem{lem}[thm]{Lemma}
 \newtheorem{conj}[thm]{Conjecture}
 \newtheorem{prop}[thm]{Proposition}
 \theoremstyle{definition}
 \newtheorem{defn}[thm]{Definition}
 \theoremstyle{remark}
 \newtheorem{exam}[thm]{Example}
 \theoremstyle{remark}
 \newtheorem{rem}[thm]{Remark}

\newtheorem*{theorem*}{Theorem}
\newtheorem*{proposition*}{Proposition}
\newtheorem*{lemma*}{Lemma}
\newtheorem*{corollary*}{Corollary}
\newtheorem*{question*}{Question}
\newtheorem*{conjecture*}{Conjecture}
\newtheorem*{claim*}{Claim}

\newtheorem*{introtheorem*}{Theorem}
\newtheorem*{introproposition*}{Proposition}
\newtheorem*{introlemma*}{Lemma}
\newtheorem*{introcorollary*}{Corollary}

\numberwithin{equation}{section}

\DeclareSymbolFont{rsfs}{U}{rsfs}{m}{n}
\DeclareSymbolFontAlphabet{\mathcal}{rsfs}

 \newcommand{\red}{\textup{red}}

 \newcommand{\Br}{\textup{Br}}

 \renewcommand{\dim}{\textup{dim}}
  \newcommand{\Spec}{\textup{Spec}}

 \newcommand{\p}{\mathfrak{p}}

 \newcommand{\A}{\mathbf{A}}
 \newcommand{\Q}{\mathbb{Q}}

 \newcommand{\Z}{\mathbb{Z}}

 \renewcommand{\O}{\mathcal{O}}
  \newcommand{\Gm}{\mathbb{G}_{m}}

 \newcommand{\br}{\mathrm{Br}}

 \newcommand{\Af}{\mathbb{A}}
 \newcommand{\Ad}{\mathbf{A}}

 \newcommand{\ok}{\mathcal{O}_k}
 \newcommand{\op}{\mathcal{O}_{\mathfrak{p}}}
 \newcommand{\ov}{\mathcal{O}_v}
\newcommand{\ord}{\mathrm{ord}}

\begin{document}

\title[]
{Local-global principle for triangularizability and diagonalizability of matrices}

\author{Kai Huang}

\address{Kai Huang \newline School of Mathematical Sciences, \newline  University of Science and Technology of China; \newline 96 Jinzhai Road, 230026 Hefei, China}

\email{hk0708@mail.ustc.edu.cn}

\author{Yufan Liu}

\address{Yufan Liu \newline School of Mathematical Sciences, \newline  University of Science and Technology of China; \newline 96 Jinzhai Road, 230026 Hefei, China}

\email{liuyufan@mail.ustc.edu.cn}


\begin{abstract}
Given a number field $k$ with the ring of integers $\ok$ and a matrix $M\in \textup{M}_{n}(\ok)$. We prove that if $\ok$ is a principal ideal domain, the local-global principle for triangularizability and diagonalizability of $M$ holds. To explain the possible failures of the local-global principle, we prove that the stratified Brauer--Manin obstruction is the only obstruction to the local-global principle for triangularizability and diagonalizability of $M$ in some special cases.
\end{abstract}

\maketitle
\section{Introduction}
The local-global principle, also known as the Hasse principle, is a central theme in number theory and expresses the philosophy that local information can determine global behavior. Their influence extends into linear algebra as well. In 1980, Guralnick established that similarity of matrices over number fields satisfies the local-global principle in \cite{guralnick}. In this paper, we examine whether diagonalizability and triangularizability of matrices also conform to a local-global principle. More precisely, we consider the following question:
\begin{question*}
Given an integer $n\geq1$ and a matrix $M\in \textup{M}_{n}(\ok)$, is $M$ diagonalizable (resp. triangularizable) over $k$ or $\ok$ provided that $M$ is diagonalizable (resp. triangularizable) over $k_{v}$ or $\ov$ for all $v\in \Omega_{k}$?
\end{question*}

When $\ok$ is a principal ideal domain, we prove that the answer to the above question is \textbf{yes} (see Proposition \ref{tp} and Proposition \ref{dp}).
However, when $\ok$ is not a PID, one can always construct counterexamples to the local-global principle for integral triangularizability and integral diagonalizability (see Remarks \ref{cextria} and \ref{de}).

For further investigation, we reformulate this problem in terms of algebraic geometry.
In fact, for a matrix $M\in \textup{M}_{n}(\ok)$, whether $M$ is triangularizable (resp. diagonalizable) over $k$ or $\ok$ can be translated into the existence of rational or integral points on the triangularization variety $X_M$ (resp. the diagonalization variety $Y_M$), which is defined at the beginning of Section \ref{to variety} as a closed subvariety of the general linear group $\mathrm{GL}_{n}$. 

For the case of diagonalization, we prove that $Y_M$ is a copy of an algebraic group for any $M$, and the Brauer--Manin obstruction is the only obstruction to the existence of rational and integral points of $Y_M$. More precisely, we have the following theorem.
\begin{thm}[Theorem \ref{bm}]
    Let $k$ be a number field and $M\in \mathrm{M}_n(\ok)$. Then 

\begin{itemize}
    \item Every connected component of the diagonalization variety $Y_{M,\red}$ satisfies weak approximation.
    \item $\overline{Y_{M,\red}(k)}= Y_{M,\red}(\A_{k})^{\Br}_{\bullet}$. In particular, the Brauer--Manin obstruction is the only obstruction to the existence of integral or rational points of $Y_{M,\red}$.
\end{itemize}
\end{thm}

However, for the case of triangularization, the corresponding varieties become quite subtle, its structure is extremely complicated. The Brauer--Manin obstruction may be insufficient, so we introduce the so-called stratified Brauer--Manin obstruction (see Definition~\ref{def str}). Based on the calculations of numerous examples, we conjecture that the stratified Brauer--Manin obstruction is the only obstruction to the existence of rational and integral points.
Moreover, we have confirmed the validity of this conjecture in several special cases:
\begin{thm}[Theorem \ref{th5.3}]
    If $M \in \mathrm{M}_{m+n}(\ok)$ is similar to $\begin{pmatrix}
            \lambda I_m & 0\\
            0 & J_{n}(\lambda)
        \end{pmatrix}$ over $k$, where $\lambda\in\ok$ and $J_{n}(\lambda)$ is the $n\times n$ Jordan block with eigenvalue $\lambda$, then we have
        \begin{itemize}
            \item Every irreducible component of the triangularization variety $X_M$ is smooth and $k$-rational.
            \item The stratified Brauer--Manin obstruction is the only obstruction to the existence of rational and integral points of $X_M$.
        \end{itemize}
\end{thm}

The outline of this article is as follows:
section \ref{notation} introduces all the notations and basic lemmas we will need in the sequel.
In section \ref{section classical}, using classical linear algebra methods, we show that if $\ok$ is a principal ideal domain, the local-global principle for triangularizability and diagonalizability of $M$ holds. Otherwise corresponding counterexamples have also been constructed there. For details, see Proposition \ref{tp}, Proposition \ref{dp} and their remarks.
To explain the failure of local-global principle, we translate the above question into the language of algebraic geometry and study the Brauer--Manin obstruction in section \ref{to variety}. Our strategy is making the most of the linear algebra information of the matrix itself to determine the geometric structure of the corresponding variety. Finally, Section \ref{stratified BM} focuses on the stratified Brauer--Manin obstruction, and its main result is Theorem \ref {th5.3}.

Finally, we further remark that the ongoing work by Mingqiang Feng and Ziyang Zhu investigates whether the local-global principle holds for integral similarity of integral matrices.

\section{Notation and Preliminaries}\label{notation}
\subsection{Notation and conventions}
We begin by introducing the notation and conventions used throughout this paper.

For a number field $k$, we denote by $\ok$ the ring of integers of $k$, by $\Omega_{k}$ the set of all places of $k$, and by
$\A_{k}$ the ring of ad\`eles of $k$.
For a finite place $\p$ of $k$, we denote by
$k_{\p}$ the completion of $k$ at $\p$, by
$\ord_{\p}(-)$ the valuation associated to $\p$, and by $\op$ the valuation ring of $k_{\p}$.
For a finite nonempty subset $S$ of $\Omega_k$ containing all the
archimedean places, we denote by $\O_{k,S}$ the ring of $S$-integers of $k$, which is precisely $\O_k$ when $S$ is the set of archimedean places. 

For a commutative unitary ring $R$, we denote 
\begin{itemize}
    \item $\textup{M}_{m,n}(R)$: set of all $m\times n$ matrices with $R$ entries,
    and  briefly $\textup{M}_{n}(R):=\textup{M}_{n,n}(R)$,
    \item $\mathrm{GL}_{n}(R)$: set of all invertible matrices in $\textup{M}_{n}(R)$,
    \item $\mathrm{SL}_{n}(R)$: set of all determinant $1$ matrices in $\mathrm{GL}_{n}(R)$,
    \item $I_{n}$: the identity matrix in $\textup{M}_{n}(R)$,
    \item $J_{m}(\lambda)$: the Jordan block of size $m$ with eigenvalue $\lambda\in R$.
\end{itemize}

For a commutative unitary ring $R$ and a matrix $M\in \textup{M}_{n}(R)$, we denote
\begin{itemize}
    \item $M^{-1}$: the inverse matrix of $M$,
    \item $M^{\mathrm{T}}$: the transpose matrix of $M$, 
    \item $M^{*}$: the adjugate matrix of $M$ (i.e., the transpose of the cofactor matrix of $M$),
    \item $\det M$: the determinant of $M$.
\end{itemize}

For a ring homomorphism $R\to S$ between commutative unitary rings and a matrix $M\in \textup{M}_{n}(R)$, we say that $M$ is diagonalizable (resp. triangularizable) over $S$ if there exists a matrix $T\in \mathrm{GL}_{n}(S)$ such that $T^{-1}MT\in \textup{M}_{n}(S)$ is diagonal (resp. upper triangular), here $M$ also means the image of $M$ under $\textup{M}_{n}(R)\to\textup{M}_{n}(S)$ by abuse of notation. 

 Let $V$ be a variety (i.e. a separated scheme of finite type) over a number field $k$, and $\mathcal{V}$ an integral model of $V$ over $\ok$ (i.e. a separated scheme of finite type over $\ok$ whose generic fiber is isomorphic to $V$). We write $V(\A_k)_{\bullet}$ for the \textbf{reduced} adelic space of $V$, where at every real place $v$ the factor $V(k_v)$ is replaced by its set of connected components, and at every complex place $v$ it is reduced to a single point.
The Brauer--Manin set $V(\A_k)^{\Br}_{\bullet}$, defined as the set of elements of $V(\A_k)_{\bullet}$ that are orthogonal to the Brauer
group $\Br(V)$ of $V$ with respect to the Brauer--Manin pairing, is indeed a closed subset. Moreover, we have a chain of inclusions $V(k)\subset V(\A_k)^{\Br}_{\bullet}\subset V(\A_k)_{\bullet}$. We say that $V$ satisfies \begin{itemize}
\item  local-global principle for rational points if
$\prod\limits_{v\in\Omega_{k}} V(k_{v})\neq \varnothing$ implies $V(k)\neq\varnothing$.
\item local-global principle for integral points (with respect to $\mathcal{V}$) if $\prod\limits_{v\in\Omega_{k}} \mathcal{V}(\ov)\neq \varnothing$ implies $\mathcal{V}(\ok)\neq\varnothing$.
\item weak approximation if the image of the diagonal map $V(k)\to \prod\limits_{v\in\Omega_{k}} V(k_{v})$ is dense.
\item strong approximation if the image of the diagonal map $V(k)\to V(\A_k)_{\bullet}$ is dense.
 \end{itemize}

 We say that the Brauer--Manin obstruction is the only
obstruction to the existence of rational (resp. integral) points of $V$ if $V(\A_k)^{\Br}_{\bullet}\neq \varnothing$ implies $V(k)\neq \varnothing$ (resp. $V(\A_k)^{\Br}_{\bullet}\cap\prod\limits_{v\in \Omega_k}\mathcal{V}(\ov)\neq \varnothing$ implies $\mathcal{V}(\ok)\neq \varnothing$).
\subsection{Preliminaries}

In this subsection, we will present some auxiliary results from the literature that will be used in Section \ref{to variety} and \ref{stratified BM}.

\begin{lem}[{\cite[Thmeorem C]{SZ14} and \cite[Theorem 3.1]{Lv20}}]\label{product}
    Let $X$ and $Y$ be smooth and geometrically integral varieties over a number field $k$, then we have
    \[
    (X\times Y)(\Ad_k)_{\bullet}^{\br}=X(\Ad_k)_{\bullet}^{\br}\times Y(\Ad_k)_{\bullet}^{\br}.
    \]
    In particular, if $\overline{X(k)}=X(\Ad_k)_{\bullet}^{\br}$ and $\overline{Y(k)}=Y(\Ad_k)_{\bullet}^{\br}$, then we have $\overline{(X\times Y)(k)}=(X\times Y)(\Ad_k)_{\bullet}^{\br}$.
\end{lem}

\begin{lem}[{\cite[Th\'eor\`eme 5.1]{Dem}}]\label{algebraic group}
    Let $G$ be a connected linear algebraic group over a number field $k$, then we have $\overline{G(k)}=G(\Ad_k)_{\bullet}^{\br}$.
\end{lem}

\begin{lem}[{\cite[Lemma 1.8]{HW16}}]\label{codim 2}
    Let $n\geq1$ be an integer, let $k$ be a number field and let $F\subseteq \Af_{k}^{n}$ be a closed subscheme of codimension at least $2$. Then the open subscheme $X=\Af_{k}^{n}\setminus F$ satisfies that $\overline{X(k)}=X(\Ad_k)_{\bullet}$.
\end{lem}

\begin{rem}
    In fact, it is shown in \cite{HW16} that $X$ satisfies strong approximation off a single place $v_0$. 
    However, their argument (reducing to the strong approximation on an affine line) remains valid for the reduced adelic space as defined here.
\end{rem}

\begin{lem}[{\cite[Proposition 13.3.21]{bgg}}]\label{disjoint}
    Let $X=\bigsqcup_{i=1}^n X_i$ be a finite disjoint union of varieties defined over a number field $k$. Then
\[
    X(\Ad_k)_\bullet^\br \;=\; \bigsqcup_{i=1}^n X_i(\Ad_k)_\bullet^\br.
\]
\end{lem}

\begin{rem}
    The reason we reduce each factor $X(k_v)$ to a single point at complex places in the definition of the reduced adelic space $X(\Ad_k)_\bullet$ is precisely to obtain the above lemma; otherwise, the components at complex places of points in the Brauer--Manin set could be chosen arbitrarily.

\end{rem}

\section{Classical results for triangularization and diagonalization}\label{section classical}

\subsection{Triangularization}
\begin{lem}\label{lem}
Let $R$ be a principal ideal domain with fraction field $K$.
\begin{enumerate}
    \item If $n\geq2$ and
$r_{1},...,r_{n}\in R$ such that $\mathrm{gcd}(r_{1},...,r_{n})=1$, then there exists $M\in\mathrm{GL}_{n}(R)$ such that $\det M=1$ and $\begin{pmatrix}
r_1 & \cdots &r_n
\end{pmatrix}^\mathrm{T}$ is the first column of $M$.
   \item A matrix $M\in \textup{M}_{n}(R)$ is triangularizable over $R$ if and only if it is triangularizable over $K$.
\end{enumerate}
\end{lem}
\begin{proof}
    \ 
\begin{enumerate}[leftmargin=*, labelsep=0.5em, align=left]
\item This is just the unimodular completion theorem, cf. \cite{Unimodular}.

\item As $\mathrm{GL}_{n}(R)\subset \mathrm{GL}_{n}(K)$, the ``only if'' part is obvious. Now we prove the ``if'' part by induction on $n$. If $n=1$, the case is trivial.

If $n\geq2$, $M$ is triangularizable over $K$ implies the characteristic
polynomial $f_M(\lambda)\in R[\lambda]$ of $M$ splits into linear factors, saying $$f_M(\lambda)=\prod\limits_{i=1}^{m}(\lambda-\lambda_{i})^{s_{i}}, \text{ where }\lambda_{i}\in R.$$ 
We can find $r_{1},...,r_{n}\in R$ such that $\mathrm{gcd}(r_{1},...,r_{n})=1$ and $$M\begin{pmatrix}
r_1\\
\vdots\\
r_n
\end{pmatrix}=\lambda_{1}\begin{pmatrix}
r_1\\
\vdots\\
r_n
\end{pmatrix}.$$ By the result of (1) there exists $T\in\mathrm{GL}_{n}(R)$ such that $\det T=1$ and $\begin{pmatrix}
r_1&\cdots&r_n\\

\end{pmatrix}^\mathrm{T}$ is the first column of $T$. 

Now $T^{-1}MT=\begin{pmatrix}
\lambda_1 & *\\
0& M'
\end{pmatrix}$, where $M'\in \textup{M}_{n-1}(R)$. Moreover, the characteristic polynomial of $M'$ must be $f_{M'}(\lambda)=(\lambda-\lambda_{1})^{s_{1}-1}\cdot\prod\limits_{i=2}^{m}(\lambda-\lambda_{i})^{s_{i}}$, hence $M'$ is triangularizable over $K$. By induction hypothesis there exists $P\in \mathrm{GL}_{n-1}(R)$ such that $P^{-1}M'P$ is upper triangular. Hence $\begin{pmatrix}
1 & 0\\
0& P
\end{pmatrix}^{-1}T^{-1}MT\begin{pmatrix}
1 & 0\\
0& P
\end{pmatrix}=\begin{pmatrix}
1 & 0\\
0& P^{-1}
\end{pmatrix}\begin{pmatrix}
\lambda_1 & *\\
0& M'
\end{pmatrix}\begin{pmatrix}
1 & 0\\
0& P
\end{pmatrix}=\begin{pmatrix}
\lambda_1 & *\\
0& P^{-1}M'P
\end{pmatrix}$ is upper triangular.
\end{enumerate}
\end{proof}

\begin{lem}\label{mil}
Let $k$ be a number field and $M\in \textup{M}_{n}(k)$. If all eigenvalues of the image of $M$ in $\textup{M}_{n}(k_{\p})$ belong to $k_{\p}$ for all but finitely many finite places $\p$ (or even weaker, all eigenvalues of the image of $M$ in $\textup{M}_{n}(\op/\p\op)$ belong to $\op/\p\op$ for all but finitely many finite places $\p$), then all eigenvalues of $M$ belong to $k$.
\end{lem}
\begin{proof}
Let $f(\lambda)=\mathrm{det}(\lambda I-M)\in k[\lambda]$ be the characteristic polynomial of $M$. If all eigenvalues of the image of $M$ in $\textup{M}_{n}(\op/\p\op)$ belong to $\op/\p\op$ for a finite place $\p$ such that $M\in \textup{M}_{n}(\op)$, then the image of $f(\lambda)$ in $(\op/\p\op)[\lambda]$ splits into linear factors. Now according to Chebotarev density theorem (applied to the splitting field of $f$), $f(\lambda)$ itself splits into linear factors.
\end{proof}

\begin{prop}\label{tp}
Let $k$ be a number field, $S$ a finite nonempty subset of $\Omega_k$ containing all the archimedean places.

Given an integer $n\geq1$ and a matrix $M\in \textup{M}_{n}(\O_{k,S})$, we consider the following statements about $M$:
\begin{enumerate}
    \item $M$ is triangularizable over $\O_{k,S}$;
    \item $M$ is triangularizable over $\op$ for any place $\p\in\Omega_{k}\setminus S$;
    \item $M$ is triangularizable over $\op/\p^{i}\op$ for any place $\p\in\Omega_{k}\setminus S$ and any $i\geq1$;
    \item $M$ is triangularizable over $\op/\p\op$ for any place $\p\in\Omega_{k}\setminus S$;
    \item $M$ is triangularizable over $k$;
    \item $M$ is triangularizable over $k_\p$ for any place $\p\in\Omega_{k}\setminus S$.    
\end{enumerate}
The relation $(1)\Rightarrow (2)\Leftrightarrow(3)\Leftrightarrow (4)\Leftrightarrow(5)\Leftrightarrow(6)$ holds.

In addition, if $\O_{k,S}$ is a principal ideal domain (such $S$ always exists since the ideal class group of $\ok$ is finite), then the relation $(1)\Leftrightarrow(2)\Leftrightarrow(3)\Leftrightarrow(4)\Leftrightarrow(5)\Leftrightarrow(6)$ holds.
\end{prop}
\begin{proof}
We only prove the case where $\O_{k,S}$ is a principal ideal domain, the other case is similar.

All right arrows ``$\Rightarrow$'' are just by definition except  $(4)\Rightarrow(5)$.

Recall that a matrix over a field is triangularizable if and only if its characteristic polynomial splits into linear factors.

$(4)\Rightarrow(5)$: This follows from Lemma \ref{mil}.

$(6)\Rightarrow(5)$: This follows from Lemma \ref{mil}.

$(5)\Rightarrow(1)$: This follows from Lemma \ref{lem}(2).
\end{proof}

\begin{rem}\label{cextria}
    If $\ok$ is a principal ideal domain, then we take $S$ as the set of archimedean places, and thereby conclude that the local-global principle for triangularizability holds.

    If $\ok$ is no longer a principal ideal domain, there always exists $M\in \mathrm{M}_n(\ok)$ such that it is triangularizable over $k$ (and thus triangularizable over $\op$ for all finite place $\p$) but not triangularizable over $\ok$.

    To construct such $M$, we take nonzero elements $a,b\in \ok$ such that $(a)+(b)$ is not a principal ideal in $\ok$ (such $a,b$ always exist since $\ok$ is not a principal ideal domain).
    Then for any $\delta\in k^{\times}$ such that $a/\delta,b/\delta\in \ok$, we have $(a/\delta)+(b/\delta)\neq \ok$ is a proper ideal of $\ok$.
    
    We take 
    \[
        N= \begin{pmatrix}
            a & & & b\\
            b & a & & \\
            & \ddots & \ddots &\\
            & & b & a
        \end{pmatrix}.
    \]
    Then we have $0\neq\det N=a^n+(-b)^n\in \ok$, or else $(a)=(b)=(a)+(b)$ is a principal ideal, which contradicts the choice of $a,b$.
    Now we choose \textbf{distinct} $\lambda_1,\ldots,\lambda_n\in \ok$ such that $\det N$ divides all $\lambda_i$ , and take 
    \[
        M=N\begin{pmatrix}
            \lambda_1 & & \\
            & \ddots & \\
            & & \lambda_n
        \end{pmatrix}N^{-1}\in \mathrm{M}_n(\ok).
    \]
    Then $M$ is triangularizable over $k$, we will show that $M$ is not triangularizable over $\ok$.

    In fact, assume that there exists $T\in \mathrm{GL}_n(\ok)$ such that $T^{-1}MT$ is an upper triangular matrix.
    Then the first column $T_1$ of $T$ is an eigenvector of $M$ whose eigenvalue coincides with $(T^{-1}MT)_{1,1} = \lambda_i$ for some $i\in\{1,\ldots,n\}$.
    Thus $T_1=N_i/\delta$ for some $\delta\in k^{\times}$ such that $a/\delta,b/\delta\in \ok$, where $N_i$ is the $i$-th column of $N$.
    Since $\det T\in (a/\delta)+(b/\delta)\neq \ok$, we have $\det T\notin \ok^{\times}$, this contradicts $T\in \mathrm{GL}_n(\ok)$.
\end{rem}

\subsection{Diagonalization}

\begin{prop}\label{dp}
Let $k$ be a number field, $S$ a finite nonempty subset of $\Omega_k$ containing all the archimedean places.

Given an integer $n\geq1$ and a matrix $M\in \textup{M}_{n}(\O_{k,S})$, we consider the following statements about $M$:
\begin{enumerate}
    \item $M$ is diagonalizable over $\O_{k,S}$;
    \item $M$ is diagonalizable over $\op$ for any place $\p\in\Omega_{k}\setminus S$;
    \item $M$ is diagonalizable over $\op/\p^{i}\op$ for any place $\p\in\Omega_{k}\setminus S$ and any $i\geq1$;
    \item $M$ is diagonalizable over $\op/\p\op$ for any place $\p\in\Omega_{k}\setminus S$;
    \item $M$ is diagonalizable over $k$;
    \item $M$ is diagonalizable over $k_\p$ for any place $\p\in\Omega_{k}\setminus S$.     
    
\end{enumerate}
The relation $(1)\Rightarrow(2)\Leftrightarrow(3)\Rightarrow(4)\Rightarrow(5)\Leftrightarrow(6)$ holds.

In addition, if $\O_{k,S}$ is a principal ideal domain, then the relation $(1)\Leftrightarrow(2)\Leftrightarrow(3)\Rightarrow(4)\Rightarrow(5)\Leftrightarrow(6)$ holds.
\end{prop}
\begin{proof}
We only prove the case where $\O_{k,S}$ is a principal ideal domain, the other case is similar.

All right arrows ``$\Rightarrow$'' are just by definition except $(4)\Rightarrow(5)$.

$(4)\Rightarrow(5)$: By Lemma \ref{mil} all eigenvalues of $M$ belong to $k$. Assume that $M$ is not diagonalizable over $k$, then there exists $T\in \mathrm{GL}_n(k)$ such that $T^{-1}MT$ is a Jordan matrix which is not diagonalizable.
Since there are only finitely many entries in $T,M$ and $T^{-1}$, there exists a place $\p\in\Omega_{k}\setminus S$ such that all the entries of $T,M$ and $T^{-1}$ belong to $\op$. 
Let $\overline{T},\overline{M}$ and $\overline{T^{-1}}$ be the images of them under the map  $\mathrm{M}_n(\op)\to\mathrm{M}_n(\op/\p\op)$.
Then we have $\overline{T}^{-1}=\overline{T^{-1}}$ and $\overline{T}^{-1}\overline{M}\overline{T}$ is a Jordan matrix which is not diagonalizable.
So $M$ is not diagonalizable over $\op/\p\op$.
This contradicts (4).

$(6)\Rightarrow(5)$: By Lemma \ref{mil} all eigenvalues of $M$ belong to $k$. Then the conclusion is due to the fact that similarity of matrices is independent of field extensions.

$(3)\Rightarrow(2)$: We fix a place $\p\in\Omega_{k}\setminus S$ and denote $$D_{i}=\{T\in\mathrm{GL}_{n}(\op/\p^{i}\op)|\text{$T^{-1}MT\in\textup{M}_{n}(\op/\p^{i}\op)$ is diagonal}\} \text{ for $i\geq1$},$$
$$D=\{T\in\mathrm{GL}_{n}(\op)|\text{$T^{-1}MT\in\textup{M}_{n}(\op)$ is diagonal}\}.$$
It's easy to check that, as sets, $D=\mathop{\underleftarrow{\mathrm{lim}}}\limits_{i}D_{i}$. Hence if $D_{i}$ are all non-empty, then $D$ is also non-empty by \cite[Lemma in page 13]{Ser}.

$(2)\Rightarrow(1)$: By Lemma \ref{mil} all eigenvalues of $M$ belong to $k$. As $M$ is diagonalizable over $\op$ hence $k_{\p}$ for any place $\p\in\Omega_{k}\setminus S$, we have $M$ is diagonalizable over $k$ since similarity of matrices is independent of field extensions.
Assume that $\lambda_1,\ldots,\lambda_r\in \O_{k,S}$ are all eigenvalues of $M$.
Let $E_i\subset k^n$ be the eigenspace of $\lambda_i$ for each $i$, $d_i=\dim_k E_i$, then $\sum d_i=n$.
Let $E_{\p,i}:=E_i\otimes_{k}k_\p\subset k_\p^n$ for any $i$ and $\p\in\Omega_{k}\setminus S$, it is just the eigenspace of $\lambda_i$ over $k_\p$.

Let $\Xi _i:=\left\{\xi\in \mathrm{M}_{n,d_i}(\O_{k,S})| \xi_{\bullet,1},\ldots,\xi_{\bullet,d_i}\ \text{constitute a}\ k\text{-basis of }E_i\right\}$ for each $i=1,\ldots,r$. Here $\xi_{\bullet,j}$ represents the vector corresponding to the $j$-th column of $\xi$ for $j=1,\ldots,d_i$.
Let 
\[
    \Xi:=\left\{\begin{pmatrix}
\xi_1&\cdots& \xi_r
\end{pmatrix}\in \mathrm{M}_n(\O_{k,S})|\xi_i\in \Xi_i\text{ for any }i=1,\ldots,r\right\}.
\]
Here $\begin{pmatrix}
\xi_1&\cdots& \xi_r
\end{pmatrix}$ denotes the concatenation of the matrices $\xi_1,\ldots,\xi_r$, i.e., the matrix obtained by placing $\xi_1,\ldots,\xi_r$ side by side.
Then $M$ is diagonalizable over $\O_{k,S}$ iff $\Xi\cap \mathrm{GL}_n(\O_{k,S})\neq \varnothing$, i.e. there exists $\xi\in \Xi$ such that $\det \xi$ is invertible in $\O_{k,S}$. 
In this case we have $\xi\in \mathrm{GL}_n(\O_{k,S})$ and 
\[
    \xi^{-1}M\xi=\begin{pmatrix}
      \lambda_1 I_{d_1}&&\\
      &\ddots&\\
      &&\lambda_r I_{d_r}
    \end{pmatrix}.
\]

Similarly for each place $\p\in\Omega_{k}\setminus S$, we denote $$\Xi_{\p,i}:=\{\xi\in \mathrm{M}_{n,d_i}(\op)|\xi_{\bullet,1},\ldots,\xi_{\bullet,d_i}\ \text{constitute a}\ k_\p\text{-basis of }E_{\p,i}\},$$
\[
    \Xi_\p:=\left\{\begin{pmatrix}
\xi_1&\cdots&\xi_r
\end{pmatrix}\in \mathrm{M}_n(\op)|\xi_i\in \Xi_{\p,i}\text{ for any }i=1,\ldots,r\right\}.
\]
Since $M$ is diagonalizable over $\op$, there exists $\xi_\p\in \Xi_\p$ such that $\ord_\p(\det \xi_\p)=0$. 
Fix a $\xi_0\in \Xi$.
By the definition of $\Xi$ and $\Xi_\p$, there exists 
\[
    G_\p=\begin{pmatrix}
    G_{\p,1}&&\\
    &\ddots&\\
    &&G_{\p,r}
    \end{pmatrix}\in \mathrm{GL}_n(k_\p)
\] 
such that $\xi_\p= \xi_0G_\p$, where each $G_{\p,i}\in \mathrm{GL}_{d_i}(k_\p)$. 
According to the strong approximation on $k$, there exists  
\[
    G_\p'=\begin{pmatrix}
G_{\p,1}'&&\\
    &\ddots&\\
    &&G_{\p,r}'
\end{pmatrix}\in \mathrm{GL}_n(k)
\] 
such that each $G_{\p,i}'$ sufficiently closes to $G_{\p,i}$ in the $\p$-adic topology and $G_\p'\in \mathrm{M}_n(\O_{\p'})$ for each finite place $\p'\neq \p$. 
Then we have $\xi_0 G_\p'\in \Xi$ and $\ord_\p(\det\xi_0  G_\p')=0$.

As $\O_{k,S}$ is Noetherian, we have $E_i\cap \O_{k,S}^n$ is a finitely generated $\O_{k,S}$-module for any $i$. Hence  $E_i\cap \O_{k,S}^n$ is isomorphic to $\O_{k,S}^{d_i}$ as a $\O_{k,S}$-module since every finitely generated torsion-free module over a principal ideal domain is free.
Let $\vartheta^i\in \Xi_i $ such that $\vartheta_{\bullet,1}^i,\ldots,\vartheta_{\bullet,d_i}^i $ constitute a $\O_{k,S}$-basis of $E_i\cap \O_{k,S}^n$.
Let 
\(
    \vartheta=\begin{pmatrix}
    \vartheta^1&\cdots& \vartheta^r
   \end{pmatrix}
\),
then for any $\xi\in \Xi$, we have $\xi=\vartheta G$ for some $G\in \mathrm{M}_n(\O_{k,S})$, it implies that $\frac{\det \xi}{\det \vartheta}\in \O_{k,S}$, i.e. $\ord_\p(\det \vartheta)\leq \ord_\p(\xi)$ for any place $\p\in\Omega_{k}\setminus S$.
Because $\xi_0 G_\p'\in \Xi$ and $\ord_\p(\det \xi_0 G_\p')=0$, we conclude that $\ord_\p(\det \vartheta)=0$ for any place $\p\in\Omega_{k}\setminus S$, i.e. $\det \vartheta$ is invertible in $\O_{k,S}$.
This completes the proof.
\end{proof}

It turns out that the opposite directions of all unidirectional arrows above are false. An example is given as follows.
\begin{exam}
For matrix $M=\begin{pmatrix}
    x & a \\
    0 & y
\end{pmatrix}\in \mathrm{M}_n(\Z)$ such that $x-y\neq0$ and $a\neq0$,
we have    \begin{itemize}
\item $M$ is diagonalizable over $\Z$ iff $x-y|a$;
\item $M$ is diagonalizable over $\Z_{p}$ iff $\ord_{p}(x-y)\leq\ord_{p}(a)$;
\item $M$ is diagonalizable over $\Z/p^{k}\Z$ iff $p^{l}|x-y$ implies $p^{min\{l,k\}}|a$ for any $l\geq1$;
\item $M$ is diagonalizable over $\Q$.
\end{itemize}
\end{exam}

\begin{rem}\label{de}    
    If $\ok$ is a principal ideal domain, then we take $S$ as the set of archimedean places, and thereby conclude that the local-global principle for diagonalizability holds.

    If $\ok$ is no longer a principal ideal domain, there always exists $M\in \mathrm{M}_n(\ok)$ such that $M$ is diagonalizable over $\op$ for all finite place $\p$ but not diagonalizable over $\ok$.
    The construction here is more complicated than that in Remark \ref{cextria}, because constructing a such $M$ that is diagonalizable over all $\p$ is no longer just a matter of constructing one that is diagonalizable over $k$.

    In the remainder of this remark, all primes with different subscripts are assumed to be \textbf{distinct}.

    To construct such $M$,  we take a prime ideal $\p_0\subset\ok$ which is not principal.
    Then there exists $m\in \Z_{>1}$ such that $\p_0^m$ is principal since the ideal class group of $k$ is finite.
    According to the strong approximation on $k$ (c.f. \cite[Theorem 13.1.1]{bgg}), one can choose $a\in \ok$ such that $(a)=\p_0\p_a$ for some prime ideal $\p_a$.
    Then $\p_a$ is not principal but $\p_a^m$ is principal.
    Assume that $\p_a^m=(a_0)$ for some $a_0\in \ok$.
    Choose $s\in \ok$ such that $(s)=\p_a\p_s$ for some prime ideal $\p_s$.
    By the strong approximation on $k$, one can choose $r,t\in \ok$ such that $(r)=\p_a^{m-1}\p_r$ and $(t)=\p_a^{m-1}\p_t$ for some prime ideals $\p_r,\p_t$.
    Then we have 
    \[
        (ar)+(st)=\p_0\p_a^m\p_r+\p_a^m\p_s\p_t=\p_a^m=(a_0),
    \]
    so there exists $\alpha,\beta\in \ok$ such that $ar\alpha+st\beta=a_0$.
    
    Now let 
    \[
        T_0=\begin{pmatrix}
            a & -t\beta & \\
            s & r\alpha & \\
              & & I_{n-2}
        \end{pmatrix}\in \mathrm{M}_n(\ok).
    \]
    Choose \textbf{distinct} $\lambda_1,\ldots,\lambda_n\in \ok$ which are divisible by $a_0$, then take
    \[
        M=T_0\begin{pmatrix}
               \lambda_1 &&\\
               &\ddots&\\
               &&\lambda_n 
           \end{pmatrix}T_0^{-1}\in \mathrm{M}_n(\ok).
    \]
    since $\det T_0=a_0$.

    Now we will prove that $M$ is diagonalizable over all $\op$ but not diagonalizable over $\ok$.
    Since $\det T_0=a_0$ is a unit at any place $\p\neq \p_a$, $M$ is diagonalizable over such $\op$.
    Let $\pi$ be a uniformizer of $\mathcal{O}_{\p_a}$, we take
    \[
        T_0'=\begin{pmatrix}
            a/\pi & -t\beta/\pi^{m-1} &\\
            s/\pi & r\alpha/\pi^{m-1} &\\
            && I_{n-2}
        \end{pmatrix}\in \mathrm{M}_n(\mathcal{O}_{\p_a}).
    \]
    Now $\det T_0'=a_{0}/\pi^{m}$ is a unit in $\mathcal{O}_{\p_a}$ which implies $T_0'\in \mathrm{GL}_n(\mathcal{O}_{\p_a})$.
    Since $T_0'^{-1}MT_0'=T_0^{-1}MT_0$ is diagonal, the matrix $M$ should be diagonalizable over $\mathcal{O}_{\p_a}$.
    
    Now we only need to prove that $M$ is not diagonalizable over $\ok$.
    Let $A$ be the eigenspace of eigenvalue $\lambda_1$ and $A_0=A\cap \ok^n$.
    For any $\lambda \in k^{\times}$ such that $\lambda(a,s,0,\ldots,0)\in A_0$ (i.e. $\lambda a,\lambda s\in \ok$), we have $\ord_{\p_a}(\lambda)\geq -1$ and $\ord_{\p}(\lambda)\geq 0$ for any $\p\neq \p_a$.
    
    If $\ord_{\p_a}(\lambda)\geq 0$, then $(\lambda a)+(\lambda s)\subseteq \p_a$.
    
    If $\ord_{\p_a}(\lambda)=-1$, then there exist a prime $\p\neq \p_a$ such that $\ord_\p(\lambda)>0$ since $\p_a$ is not principal, then $(\lambda a)+(\lambda s)\subseteq \p$.
    
    Therefore, for any $\lambda(a,s,0,\ldots,0)\in A_0$, we have $(\lambda a)+(\lambda s)\varsubsetneqq \ok$ is a proper ideal.
    Hence for any $T\in \mathrm{M}_n(\ok)\cap \mathrm{GL}_n(k)$ such that $T^{-1}MT$ is diagonal, we have $(\det T)\varsubsetneqq \ok$, i.e. $T\notin \mathrm{GL}_n(\ok)$.
    Then $M$ is not diagonalizable over $\ok$.
\end{rem}

\section{Reformulating problems via the local-global principle for varieties}\label{to variety}
In this section we translate the problems of triangularization and diagonalization of $M$ into the local-global principle for some special types of varieties.
Let $k$ be a number field and $M\in \mathrm{M}_n(\ok)$. We fix an algebraic closure $\overline{k}$ of $k$. Let $T=(T_{i,j})_{1\leq i,j\leq n}$ be a matrix with adjugate matrix $T^{*}$, where the subscript $i,j$ means the $(i,j)$-entry of $T$. We denote
\begin{itemize}
\item \textbf{Triangularization variety} $X_{M}\subset\mathbb{A}_{k}^{n^{2}+1}=\Spec \ k[T_{i,j},w]$: 
the closed subscheme defined by $$\text{$w\cdot\det(T)=1$ and $(T^{*}MT)_{i,j}=0$ for $i>j$.} $$

Also its integral model $\mathcal{X}_{M} \subset\mathbb{A}_{\ok}^{n^{2}+1}=\Spec \ok[T_{i,j},w]$ defined by the same equations.
\item \textbf{Diagonalization variety} $Y_{M}\subset\mathbb{A}_{k}^{n^{2}+1}=\Spec \ k[T_{i,j},w]$:
the closed subscheme defined by $$\text{$w\cdot\det(T)=1$ and $(T^{*}MT)_{i,j}=0$ for $i\neq j$.}$$

Also its integral model $\mathcal{Y}_{M} \subset\mathbb{A}_{\ok}^{n^{2}+1}=\Spec \ok[T_{i,j},w]$ defined by the same equations.
\end{itemize}

By definition for every field extension $L/k$, there is a bijection between the set of $L$-points of $X_M$ (resp. $Y_M$) and the set of transition matrices of $M$ that achieve triangularization (resp. diagonalization) over $L$, thus the problem about triangularization or diagonalization of $M$ is transformed into a problem about arithmetic properties of these varieties. 
\begin{thm}\label{hasse}
Let $k$ be a number field and $M\in \mathrm{M}_n(\ok)$.
Then $X_{M}$ and $Y_{M}$ both satisfy local-global principle for rational points.

Moreover if $\ok$ is a principal ideal domain, then $X_{M}$ and $Y_{M}$ both satisfy local-global principle for integral points.
\end{thm}
\begin{proof}
We directly apply Proposition \ref{tp}, Remark \ref{cextria}, Proposition \ref{dp}, Remark \ref{de} to conclude. 
\end{proof}

In the following, we study the strong approximation and Brauer--Manin obstructions of these varieties, aiming to explain the possible failures of the local-global principle. In fact, it is better to consider $X_{M,\red}$ (resp. $ Y_{M,\red}$), the reduced scheme associated to $X_{M}$ (resp. $Y_{M}$), and similarly for their integral models. As before, we will discuss cases of triangularization and diagonalization separately.

\subsection{Diagonalization}

\

\begin{thm}\label{bm}
Let $k$ be a number field and $M\in \mathrm{M}_n(\ok)$. Then we have

\begin{itemize}
    \item Every connected component of $Y_{M,\red}$ satisfies weak approximation.
    \item $\overline{Y_{M,\red}(k)}= Y_{M,\red}(\A_{k})^{\Br}_{\bullet}$. In particular, the Brauer--Manin obstruction is the only obstruction to the existence of integral or rational points of $Y_{M,\red}$.
\end{itemize}
\end{thm}
\begin{proof}
We can assume $\prod\limits_{v\in\Omega_{k}}Y_{M}(k_{v})\neq\varnothing$.
As $Y_{M}$ satisfies local-global principle for rational points by Theorem \ref{hasse}, $\prod\limits_{v\in\Omega_{k}}Y_{M}(k_{v})\neq\varnothing$ implies $Y_{M}(k)\neq\varnothing$. Hence $M$ is diagonalizable over $k$. 

More explicitly, $M$ is similar to $M'=\mathrm{diag}(\lambda_1 I_{n_1},\lambda_2 I_{n_2},\ldots,\lambda_r I_{n_r})\in \mathrm{M}_n(k)$ over $k$. Here $\lambda_{1},\ldots,\lambda_{r}\in \ok$ are distinct and $n_{1}+\cdots+n_{r}=n$. Now $Y_{M}\cong Y_{M'}$ as $k$-varieties since similar matrices correspond to isomorphic varieties, and the isomorphism is given by a coordinates change of $\mathbb{A}_{k}^{n^{2}+1}=\Spec \ k[T_{i,j},w]$. So we can assume $M=M'$.
    
Let $\mathfrak{s}:=\{\sigma:\{1,\ldots,n\}\to\{\lambda_1,\ldots,\lambda_r\}\mid \#\sigma^{-1}(\lambda_i)=n_i\ \text{for each}\ i=1,\ldots,r\}$. Let $\mathfrak{S}_n$ denote the symmetric group on $\{1,\ldots,n\}$, acting on $\mathfrak{s}$ by $(\tau\cdot\sigma)(i):=\sigma(\tau^{-1}(i))$ for $\tau\in\mathfrak{S}_n$. This action is transitive, and the stabilizer of any $\sigma\in\mathfrak{s}$ is the subgroup $\prod_{i=1}^r\mathfrak{S}_{n_i}$ permuting the indices within each fiber $\sigma^{-1}(\lambda_i)$; hence $\mathfrak{s}\cong\mathfrak{S}_n/\prod_{i=1}^r\mathfrak{S}_{n_i}$ as sets, and in particular $\#\mathfrak{s}=\frac{n!}{n_{1}!\cdots n_{r}!}$.

Let $Y_{M,\sigma}$ be the closed subscheme of $\Spec\ k[T_{i,j},w]$ which is defined by
\[
    w\cdot \det(T)=1,\quad(T^{*}MT)_{i,j}=\begin{cases}
        0 & \text{for $i\neq j$},\\
        \det(T)\cdot \sigma(i) & \text{for $i=j$}.
    \end{cases}
\]
It is also a closed subscheme of $Y_M$, and we have $Y_M=\bigsqcup_{\sigma\in \mathfrak{s}} Y_{M,\sigma}$ as topological spaces by checking their sets of $\overline{k}$-points. Hence $Y_{M,\red}=\bigsqcup_{\sigma\in \mathfrak{s}} Y_{M,\sigma,\red}$ as $k$-varieties.

\textbf{Step $1$}: 
    For any $\sigma,\sigma'\in \mathfrak{s}$, we have $Y_{M,\sigma,\red}\simeq Y_{M,\sigma',\red}$ as $k$-varieties.
    
We only need to prove the case where $\sigma'=\tau\cdot\sigma$ for an adjacent transposition $\tau=(t,\,t+1)\in\mathfrak{S}_n$ with $1\leq t\leq n-1$, since such transpositions generate $\mathfrak{S}_n$. Explicitly, this means $\sigma'(t)=\sigma(t+1)$, $\sigma'(t+1)=\sigma(t)$, and $\sigma'(i)=\sigma(i)$ for $i\neq t,t+1$. Write $\sigma(t)=\lambda_\alpha$ and $\sigma(t+1)=\lambda_\beta$ with $\alpha\neq\beta$.

Note that the matrix $$\begin{pmatrix}
            I_{t-1} & & & \\
             & 0 & 1 & \\
             & 1 & 0 & \\
             & & & I_{n-t-1} 
        \end{pmatrix}\in \mathrm{GL}_{n}(k)$$
induces a $k$-automorphism of $\Spec\ k[T_{i,j},w]$ which is given by:
    \begin{align*}
        w&\mapsto w\\
        (T_{i,j})_{1\leq i,j\leq n} & \mapsto (T_{i,j})_{1\leq i,j\leq n}\cdot\begin{pmatrix}
            I_{t-1} & & & \\
             & 0 & 1 & \\
             & 1 & 0 & \\
             & & & I_{n-t-1} 
        \end{pmatrix},
    \end{align*}
we claim that this automorphism induces an isomorphism from $Y_{M,\sigma,\red}$ to $Y_{M,\sigma',\red}$. It suffices to show that it induces a bijection from $Y_{M,\sigma,\red}(\overline{k})$ to $Y_{M,\sigma',\red}(\overline{k})$
since the reduced induced closed
subscheme structure is unique. The remains are straightforward calculations of linear algebra, note that \[
        \begin{pmatrix}
            0 & 1\\
            1 & 0
        \end{pmatrix}^{-1}
        \begin{pmatrix}
            \lambda_\alpha & 0\\
            0 & \lambda_\beta
        \end{pmatrix}
        \begin{pmatrix}
            0 & 1\\
            1 & 0
        \end{pmatrix}=
        \begin{pmatrix}
            \lambda_\beta & 0\\
           0 & \lambda_\alpha
        \end{pmatrix}.
    \]

\textbf{Step $2$}: 
For $\sigma_0\in\mathfrak{s}$ such that $(\sigma_0(1),\ldots,\sigma_0(n))=(\overbrace{\lambda_{1},\ldots,\lambda_{1}}^{n_{1}},\overbrace{\lambda_{2},\ldots,\lambda_{2}}^{n_{2}},\ldots,\overbrace{\lambda_{r},\ldots,\lambda_{r}}^{n_{r}})$, we have
$Y_{M,\sigma_0,\red}\cong\mathrm{GL}_{n_{1}}\times\cdots\times\mathrm{GL}_{n_{r}}$ as $k$-varieties.

For $T\in\mathrm{GL}_n(\overline{k})$ with $w=\det(T)^{-1}$, the defining equations of $Y_{M,\sigma_0,\red}$ say that $T^*MT=\det(T)\cdot M$, equivalently $T^{-1}MT=M$, i.e.\ $T$ commutes with $M=\mathrm{diag}(\lambda_1 I_{n_1},\ldots,\lambda_r I_{n_r})$. Since $\lambda_1,\ldots,\lambda_r$ are distinct, this holds if and only if $T$ is block diagonal, $T=\mathrm{diag}(A_1,\ldots,A_r)$ with $A_s\in\mathrm{GL}_{n_s}(\overline{k})$ for $1\leq s\leq r$. Thus $Y_{M,\sigma_0,\red}(\overline{k})$ is identified with $\mathrm{GL}_{n_1}(\overline{k})\times\cdots\times\mathrm{GL}_{n_r}(\overline{k})$ via $T\mapsto(A_1,\ldots,A_r)$, which is an isomorphism of reduced closed subvarieties of $\Spec\,k[T_{i,j},w]$, both being cut out by $w\cdot\det(T)=1$ together with $T_{i,j}=0$ for $(i,j)$ outside the diagonal blocks $B_s=\{(i,j)\mid n_1+\cdots+n_{s-1}+1\leq i,j\leq n_1+\cdots+n_s\}$, $1\leq s\leq r$.

\textbf{Step $3$}: In summary, $Y_{M,\red}$ has $\frac{n!}{n_{1}!\cdots n_{r}!}$ connected components, and every connected component is isomorphic to $\mathrm{GL}_{n_{1}}\times\cdots\times\mathrm{GL}_{n_{r}}$ as $k$-varieties.
Hence every connected component of $Y_{M,\red}$ satisfies weak approximation by \cite[Theorem 13.1.2]{bgg}.

To show $\overline{Y_{M,\red}(k)}=Y_{M,\red}(\A_{k})^{\Br}_{\bullet}$, note that $$Y_{M,\red}(k)=\bigsqcup_{\sigma\in \mathfrak{s}} (Y_{M,\sigma,\red}(k))\cong\bigsqcup_{\sigma\in \mathfrak{s}} (\mathrm{GL}_{n_1}(k)\times\cdots\times \mathrm{GL}_{n_r}(k))$$ $$Y_{M,\red}(\A_{k})^{\Br}_{\bullet}=\bigsqcup_{\sigma\in \mathfrak{s}} (Y_{M,\sigma,\red}(\A_{k})^{\Br}_{\bullet})\cong\bigsqcup_{\sigma\in \mathfrak{s}} (\mathrm{GL}_{n_1}(\A_{k})^{\Br}_{\bullet}\times\cdots\times \mathrm{GL}_{n_r}(\A_{k})^{\Br}_{\bullet})$$ by Lemma \ref{product} and Lemma~\ref{disjoint}, where the last bijections are induced by the isomorphism of varieties in Step~2. 
Hence we reduce to the case of connected linear algebraic groups, this is done by Lemma \ref{algebraic group}.
\end{proof}

\begin{rem}
Keep the notation above, where $M=\mathrm{diag}(\lambda_1 I_{n_1},\ldots,\lambda_r I_{n_r})$. The results of Steps~1 and~2 can be rephrased in terms of group actions.
Let $G:=\prod_{i=1}^r\mathrm{GL}_{n_i}\times\mathfrak{S}_n$, with identity component $G^\circ=\prod_{i=1}^r\mathrm{GL}_{n_i}\times\{1\}$. The group $G$ acts on $Y_M$ by $(A,\tau)\cdot T:=ATP_\tau^{-1}$, where $A=(A_1,\ldots,A_r)\in\prod_{i=1}^r\mathrm{GL}_{n_i}$ is embedded in $\mathrm{GL}_n$ as the block diagonal matrix $\mathrm{diag}(A_1,\ldots,A_r)$, and $P_\tau$ denotes the permutation matrix associated to $\tau$. This action is transitive on $Y_M$ and preserves $Y_{M,\red}$.
The stabilizer of $I_n$ in $G$ is $\{(P_\tau,\tau)\mid\tau\in\prod_{i=1}^r\mathfrak{S}_{n_i}\}\cong\prod_{i=1}^r\mathfrak{S}_{n_i}$. In particular, the stabilizer of $I_n$ in $G^\circ$ is trivial. Hence the orbit of $I_n$ under $G^\circ$ is isomorphic to $G^\circ\cong\prod_{i=1}^r\mathrm{GL}_{n_i}$ as $k$-varieties.
Since $G=G^\circ\cdot(\{I_n\}\times\mathfrak{S}_n)$, every $G^\circ$-orbit on $Y_{M,\red}$ arises in the following way. For $\tau\in\mathfrak{S}_n$, the element $(I_n,\tau)\in G$ carries the orbit of $I_n$ under $G^\circ$ to another $G^\circ$-orbit, and this gives a $k$-isomorphism between the two orbits. Two elements $\tau,\tau'\in\mathfrak{S}_n$ yield the same orbit if and only if $\tau'^{-1}\tau\in\prod_{i=1}^r\mathfrak{S}_{n_i}$. Thus the $G^\circ$-orbits are indexed by $\mathfrak{S}_n/\prod_{i=1}^r\mathfrak{S}_{n_i}\cong\mathfrak{s}$.
Their disjoint union is all of $Y_{M,\red}$, and each orbit is isomorphic to $\prod_{i=1}^r\mathrm{GL}_{n_i}$ as $k$-varieties. These orbits are therefore exactly the connected components of $Y_{M,\red}$.
\end{rem}

\subsection{Triangularization}

\

Let $k$ be a number field and $M\in \mathrm{M}_n(\ok)$. We first figure out the structure of $X_{M,\red}$ under the condition $\prod\limits_{v\in\Omega_{k}}X_{M}(k_{v})\neq\varnothing$.
As $X_{M}$ satisfies local-global principle for rational points by Theorem \ref{hasse}, $\prod\limits_{v\in\Omega_{k}}X_{M}(k_{v})\neq\varnothing$ implies $X_{M}(k)\neq\varnothing$. Hence $M$ is triangularizable over $k$. 

More explicitly, $M$ is similar to $M'=\mathrm{diag}(J(\lambda_1),J(\lambda_2)\ldots ,J(\lambda_r))$ over $k$. Here $\lambda_1,...,\lambda_{r}\in \ok$ are distinct, $J(\lambda_{i})\in \mathrm{M}_{n_{i}}(k)$ is a Jordan canonical form with the single eigenvalue $\lambda_i$, where $n_{1}+...+n_{r}=n$. Now $X_{M}\cong X_{M'}$ as $k$-varieties since similar matrices correspond to isomorphic varieties, and the isomorphism is given by a coordinates change of $\mathbb{A}_{k}^{n^{2}+1}=\Spec \ k[T_{i,j},w]$. So we can assume $M=M'$.

Let $\mathfrak{s}:=\{\sigma:\{1,\ldots,n\}\to\{\lambda_1,\ldots,\lambda_r\}\mid \#\sigma^{-1}(\lambda_i)=n_i\ \text{for each}\ i=1,\ldots,r\}$, together with the action of $\mathfrak{S}_n$ on $\mathfrak{s}$, be defined as in the proof of Theorem~\ref{bm}. Let $X_{M,\sigma}$ be the closed subscheme of $\Spec\ k[T_{i,j},w]$ which is defined by
\[
    w\cdot\det(T)=1,(T^{*}MT)_{i,j}=\begin{cases}
        0 & \text{for $i> j$},\\
        \det(T)\cdot \sigma(i) & \text{for $i=j$}.
    \end{cases}
\]
It is also a closed subscheme of $X_M$, and we have $X_M=\bigsqcup_{\sigma\in \mathfrak{s}} X_{M,\sigma}$ as topological spaces by checking their sets of $\overline{k}$-points. Hence $X_{M,\red}=\bigsqcup_{\sigma\in \mathfrak{s}} X_{M,\sigma,\red}$ as $k$-varieties.

\begin{prop}\label{permutation}
    For any $\sigma,\sigma'\in \mathfrak{s}$, we have $X_{M,\sigma,\red}\cong X_{M,\sigma',\red}$ as $k$-varieties.    
\end{prop}

\begin{proof}
     We only need to prove the case where $\sigma'=\tau\cdot\sigma$ for an adjacent transposition $\tau=(t,\,t+1)\in\mathfrak{S}_n$ with $1\leq t\leq n-1$, since such transpositions generate $\mathfrak{S}_n$. Explicitly, this means $\sigma'(t)=\sigma(t+1)$, $\sigma'(t+1)=\sigma(t)$, and $\sigma'(i)=\sigma(i)$ for $i\neq t,t+1$. Write $\sigma(t)=\lambda_\alpha$ and $\sigma(t+1)=\lambda_\beta$ with $\alpha\neq\beta$.

We denote the polynomial $(T^{*}MT)_{t,t+1}\in k[T_{i,j}]$ by $f$, note that $$\begin{pmatrix}
            I_{t-1} & & & \\
             & fw & \frac{f^2 w^2-1}{\lambda_\beta-\lambda_\alpha} & \\
             &\lambda_\beta-\lambda_\alpha & fw & \\
             & & & I_{n-t-1} 
        \end{pmatrix}\in \mathrm{GL}_{n}(k[T_{i,j},w]).$$
This matrix induces a $k$-automorphism of $\Spec\ k[T_{i,j},w]$ which is given by:
    \begin{align*}
        w&\mapsto w\\
        (T_{i,j})_{1\leq i,j\leq n} & \mapsto (T_{i,j})_{1\leq i,j\leq n}\cdot\begin{pmatrix}
            I_{t-1} & & & \\
             & fw & \frac{f^2 w^2-1}{\lambda_\beta-\lambda_\alpha} & \\
             &\lambda_\beta-\lambda_\alpha & fw & \\
             & & & I_{n-t-1} 
        \end{pmatrix},
    \end{align*}
we claim that this automorphism induces an isomorphism from $X_{M,\sigma,\red}$ to $X_{M,\sigma',\red}$. It suffices to show that it induces a bijection from $X_{M,\sigma,\red}(\overline{k})$ to $X_{M,\sigma',\red}(\overline{k})$
since the reduced induced closed
subscheme structure is unique. The remains are straightforward calculations of linear algebra, note that \[
        \begin{pmatrix}
            fw & \frac{f^2 w^2-1}{\lambda_\beta-\lambda_\alpha}\\
            \lambda_\beta-\lambda_\alpha & fw
        \end{pmatrix}^{-1}
        \begin{pmatrix}
            \lambda_\alpha & fw\\
            0 & \lambda_\beta
        \end{pmatrix}
        \begin{pmatrix}
            fw & \frac{f^2 w^2-1}{\lambda_\beta-\lambda_\alpha}\\
            \lambda_\beta-\lambda_\alpha & fw
        \end{pmatrix}=
        \begin{pmatrix}
            \lambda_\beta & fw\\
            0 & \lambda_\alpha
        \end{pmatrix}.
    \]
\end{proof}

\begin{prop}\label{4.4}
    For any $\sigma\in \mathfrak{s}$, we have 
    \[
         X_{M,\sigma,\red}\cong \Af_k^{n(n+1)/2-\sum n_i(n_i+1)/2}\times \prod_{i=1}^r X_{J(\lambda_i),\red}.
    \]
\end{prop}

\begin{proof}
    According to Proposition \ref{permutation}, we only need to prove the case where 
    \[
    (\sigma(1),\ldots,\sigma(n))=(\overbrace{\lambda_{1},\ldots,\lambda_{1}}^{n_{1}},\overbrace{\lambda_{2},\ldots,\lambda_{2}}^{n_{2}},\ldots,\overbrace{\lambda_{r},\ldots,\lambda_{r}}^{n_{r}}).
    \]

One can easily check that the variety on the right-hand side embeds naturally into
\[
\Spec\ k[T_{i,j},w]/(w\cdot\det(T)-1; T_{i,j}\ \text{for all } (i,j)\in \bigcup_{s=1}^{r-1} B_s),
\]
where for each \(s=1,\dots,r-1\), we define
\[
B_s := \{(i,j)\mid n_1+\cdots+n_s+1 \le i \le n,\; n_1+\cdots+n_{s-1}+1 \le j \le n_1+\cdots+n_s \}.
\]
Hence both sides are reduced closed subvarieties of $\Spec \ k[T_{i,j},w]$. To complete the proof, it suffices to show they have the same set of $\overline{k}$-points. The remains are straightforward calculations of linear
algebra. Note that $X_{M,\sigma,\red}(\overline{k})\subset \mathrm{GL}_{n}(\overline{k})$ is the set of matrices of form
    \[
        \begin{pmatrix}
            T_1 & * & * \\
             & \ddots & * \\
             && T_{r}
        \end{pmatrix},
    \]
    where $T_i\in X_{J(\lambda_i),\red}(\overline{k})$.
\end{proof}
\

Based on previous propositions, we know that the arithmetic properties of $X_{M,\red}$ reduce to $X_{J(\lambda_i),\red}$.
In particular, if $\overline{X_{J(\lambda_i),\red}(k)}= X_{J(\lambda_i),\red}(\A_{k})^{\Br}_{\bullet}$ for each $1\leq i\leq r$ then $\overline{X_{M,\red}(k)}= X_{M,\red}(\A_{k})^{\Br}_{\bullet}$, hence the Brauer--Manin obstruction will be the only
obstruction to the existence of integral or rational points of $X_{M,\red}$ by the same proof of the case of diagonalization.

\begin{prop}\label{prop4.5}
In both cases below, we have  $\overline{X_{J(\lambda_i),\red}(k)}= X_{J(\lambda_i),\red}(\A_{k})^{\Br}_{\bullet}$.
\begin{itemize}
\item $J(\lambda_i)=\lambda_i I_{n_i}$, 
\item $J(\lambda_i)=J_{n_i}(\lambda_i)$ is the Jordan block of size $n_i$ with eigenvalue $\lambda_i$.
\end{itemize}
\end{prop}

\begin{proof}
If $J(\lambda_i)=\lambda_i I_{n_i}$, one checks that $X_{J(\lambda_i),\red}\cong \mathrm{GL}_{n_i}$ as $k$-varieties.

If $J(\lambda_i)=J_{n_i}(\lambda_i)$, one checks that $X_{J(\lambda_i),\red}(\overline{k})\subset \mathrm{GL}_{n_i}(\overline{k})$ is the set of upper triangular matrices, hence $X_{J(\lambda_i),\red}\cong \mathbb{G}_m^{n_i}\times \Af_k^{n_i(n_i-1)/2}$ as $k$-varieties.

So we apply Lemma \ref{product} and Lemma \ref{algebraic group} to conclude.
\end{proof}

However, the structure of $X_{J(\lambda_i),\red}$ is quite intricate in general. For example, it can be a connected but not irreducible variety, hence it will no longer have a structure of algebraic group. Now the more reasonable object is the stratified Brauer--Manin obstruction.

\section{The stratified Brauer--Manin obstruction}\label{stratified BM}
Let $V$ be a variety over a number field $k$, and $\mathcal{V}$ an integral model of $V$ over $\ok$. Consider the decomposition of $\mathcal{V}$ into its irreducible components (endowed with the reduced induced closed subscheme structure):$$\mathcal{V}=(\bigcup_i \mathcal{V}_i)\bigcup(\bigcup_j \mathcal{V}_j'),$$
where each $\mathcal{V}_i\to \Spec\ok$ is dominant but each $\mathcal{V}_j'\to \Spec\ok$ is not. Therefore $V=\bigcup_i V_i$ is the decomposition of $V$ into its irreducible components, where $V_i=\mathcal{V}_i\times_{\Spec\ok}\Spec \ k$.
\begin{defn}\label{def str}
The stratified Brauer–Manin set of $V$ is defined as $\bigcup_i V_i(\A_k)_{\bullet}^{\Br}$, and we denote it by $V(\A_k)_{\bullet}^{\mathrm{st}\br}$. 
Note that the following inclusions: $V(k)\subset V(\A_k)_{\bullet}^{\mathrm{st}\br}\subset V_{\red}(\A_k)_{\bullet}^{\Br}\subset V(\A_k)_{\bullet}$.

We shall say that the stratified Brauer--Manin obstruction is the only obstruction to the existence of rational (resp. integral) points of $V$ if $V(\A_k)_{\bullet}^{\mathrm{st}\br}\neq \varnothing$ implies $V(k)\neq \varnothing$
(resp. $V(\A_k)_{\bullet}^{\mathrm{st}\br}\bigcap\prod\limits_{v\in \Omega_k}\mathcal{V}(\ov)\neq \varnothing$ implies $\mathcal{V}(\ok)\neq \varnothing$).
\end{defn}
Now we continue the discussion on the problem of triangularization from the previous section, and we first put forward two programmatic conjectures herein.

\begin{conj}\label{co1}
Let $k$ be a number field and $M\in \mathrm{M}_n(\ok)$ with $X_M(k)\neq \varnothing$. Then every irreducible component of $X_M$ (with reduced
induced closed subscheme structure) is a smooth and $k$-rational variety.
\end{conj}

\begin{conj}\label{co2}
Let $k$ be a number field and $M\in \mathrm{M}_n(\ok)$. Then $X_{M}(k)$ is a dense subset of the stratified Brauer--Manin set of $X_M$.
In particular, the stratified Brauer--Manin obstruction is the only obstruction to the existence of rational and integral points of $X_M$.
\end{conj}

The following example demonstrates the necessity of the stratified Brauer--Manin set: there exists $M$ such that $X_M(k)$ is dense in $X_M(\A_k)^{\mathrm{st}\br}_\bullet$ but not in $X_{M,\red}(\A_k)^{\br}_\bullet$.

\begin{exam}
Let $M = \mathrm{diag}(\lambda,J_2(\lambda))$. One can check that
\[
    X_{M,\red} \;\cong\; \Spec\ k[T_{i,j}\mid_{1\leq i,j\leq 3}, w]/(w\cdot\det(T_{i,j})-1,\,
    T_{3,1},\,T_{1,1}\cdot T_{3,2}),
\]
with irreducible components
\[
    X_1 = \{ T_{3,1} = T_{1,1} = 0 \}, \quad X_2 = \{ T_{3,1} = T_{3,2} = 0 \};
\]
for the details of this isomorphism, see Remark~\ref{another def} and 
Lemma~\ref{lem:newdef} below.
Consider the $k$-points
\[
    R = \begin{pmatrix} 0&1&0\\1&0&0\\0&0&1 \end{pmatrix} \in X_1 \cap X_2, \quad
    P = \begin{pmatrix} 1&1&0\\1&0&0\\0&0&1 \end{pmatrix} \in X_2 \setminus X_1, \quad
    Q = \begin{pmatrix} 0&1&0\\1&0&0\\0&1&1 \end{pmatrix} \in X_1 \setminus X_2.
\]
The points $R$ and $P$ (resp. $R$ and $Q$) are joined by an affine line
$\gamma_P: \Af^1_k \hookrightarrow X_2$ (resp. $\gamma_Q: \Af^1_k \hookrightarrow X_1$) defined by
\[
    \gamma_P(t) = \begin{pmatrix} t&1&0\\1&0&0\\0&0&1 \end{pmatrix}, \quad
    \gamma_Q(t) = \begin{pmatrix} 0&1&0\\1&0&0\\0&t&1 \end{pmatrix}.
\]

For any $\alpha \in \Br(X_{M,\red})$, the pullbacks $\gamma_P^*(\alpha)$ and $\gamma_Q^*(\alpha)$ lie in $\Br(\mathbb{A}^1_k)=\Br(k)$. Hence the endpoint evaluations satisfy
\[
\alpha(P)=\alpha(R)=\alpha(Q)\in \Br(k),
\]
and we denote this element by $\beta$.
Now fix a finite place $v_0$ of $k$, and define an adelic point $x=(x_v)_{v\in\Omega_k}$ by setting $x_{v_0}=P$ and $x_v=Q$ for all $v\neq v_0$. For every place $v$, the local evaluation $\alpha(x_v)$ is the natural restriction of $\beta \in \Br(k)$ to $\Br(k_v)$. We conclude that
$x \in X_{M,\red}(\A_k)^{\Br}_{\bullet}$ by global reciprocity.

We claim that $x$ is not in the closure of $X_M(k)$. Choose a finite 
place $v_1 \neq v_0$, and consider the open neighborhood of $x$ defined by $T_{1,1} \neq 0$ at $v_0$, $T_{3,2} \neq 0$ at $v_1$, and $\mathcal{X}_M(\mathcal{O}_v)$ at all other places. 
If $y \in X_M(k)$ lies in this neighborhood, then $T_{1,1}(y) \neq 0$ in $k$, 
hence $T_{3,2}(y) = 0$ in $k$ since $y \in X_M$, contradicting $T_{3,2}(y) \neq 0$ at $v_1$. 
Hence $\overline{X_M(k)} \subsetneqq  X_{M,\red}(\A_k)^{\br}_\bullet$.

On the other hand, one can check that each $X_i$ is isomorphic to $\mathrm{GL}_2 \times \Gm \times \Af^2_k$, hence is a connected linear algebraic group. Therefore Lemma~\ref{algebraic group} and Lemma~\ref{product} give $\overline{X_i(k)} = X_i(\A_k)^{\br}_\bullet$ for $i = 1, 2$, which shows that $X_M(k)$ is dense in $X_M(\A_k)^{\mathrm{st}\br}_\bullet$.
\end{exam}

\begin{rem}
    Appendix \ref{explicit appendix} gives an explicit matrix $M$ for which the Brauer--Manin obstruction does obstruct the existence of integral points on $X_M$.
\end{rem}

\

According to Proposition~\ref{4.4}, it suffices to consider the situation where $M$ is a Jordan canonical form with a single eigenvalue. 

\begin{thm}\label{th5.3}
If $M\in \mathrm{M}_{m+n}(\ok)$ is similar to $\begin{pmatrix}
            \lambda I_m & 0\\
            0 & J_{n}(\lambda)
        \end{pmatrix}$ over $k$ for some $\lambda\in\ok$, then Conjecture \ref{co1} and Conjecture \ref{co2} are true.
\end{thm}

The rest of this section is devoted to the proof of Theorem \ref{th5.3}.
Without loss of generality, we may assume that $M=\mathrm{diag}(\lambda I_m,J_n(\lambda))$, since similar matrices correspond to isomorphic varieties.
The elementary cases $M = \lambda I_m$ and $M = J_{n}(\lambda)$ have already been settled in Proposition \ref{prop4.5}.
In what follows, we therefore restrict to the situation where $m \geq 1$ and $n \geq 2$.

\

For a morphism \( f: W \to V \) of \( k \)-varieties, the technique of “spreading out” (cf.~\cite[Theorem 3.2.1]{poo23}) ensures that there exists a finite set \( S \subset \Omega_k \), containing all archimedean places, and a model \( \mathcal{W} \to \mathcal{V} \) over \( \mathcal{O}_{k,S} \).

\begin{defn}\label{goodfib}
A morphism \( f: W \to V \) is called a \emph{good fibration} if there exist a finite set \( S \subset \Omega_k \), containing all archimedean places, and a model \( \mathcal{W} \to \mathcal{V} \) over \( \mathcal{O}_{k,S} \), such that:
\begin{itemize}
    \item \( f \) is smooth;
    \item for all but finitely many \( v \in \Omega_k \setminus S \), the induced map \( \mathcal{W}(\mathcal{O}_v) \to \mathcal{V}(\mathcal{O}_v) \) is surjective;
    \item for every \( x \in V(k) \), the fiber \( W_x \) satisfies \( \overline{W_x(k)} = W_x(\A_k)_\bullet \).
\end{itemize}
\end{defn}
In fact, the second condition is satisfied when all fibers of $f$ are nonempty and geometrically connected, cf. \cite[Theorem 4.5]{con12}.
In what follows, we shall always use the fact that all fibers are geometrically integral to verify the second condition stated above.

\begin{lem}\label{lem5.4}
If $f:W\to V$ is a good fibration, then
\begin{enumerate}
\item $\overline{V(k)}=V(\A_k)_{\bullet}^{\Br}$ implies $\overline{W(k)}=W(\A_k)_{\bullet}^{\Br}$,
\item $\overline{V(k)}=V(\A_k)_{\bullet}$ implies $\overline{W(k)}=W(\A_k)_{\bullet}$.
\end{enumerate}
\end{lem}
\begin{proof}
\
\begin{enumerate}[leftmargin=*, labelsep=0.5em, align=left]
    \item There exists a finite set of places $S$ containing all archimedean places such that $f$ has a model $\mathcal{W}\to\mathcal{V}$ over $\O_{k,S}$. By expanding $S$, we can assume for all $v\in \Omega_{k}\setminus S$, the map $\mathcal{W}(\ov)\to\mathcal{V}(\ov)$ is surjective.

Without loss of generality, we assume that $W(\A_k)^{\br}_\bullet$ is nonempty.  
Then take a nonempty open subset $\Delta$ of $W(\A_k)_{\bullet}^{\Br}$, we want to show that $W(k)\cap \Delta\neq\varnothing$. 
We may assume $\Delta$ is of the form $(\prod\limits_{v\in T}U_v\times \prod\limits_{v\notin T}\mathcal{W}(\ov))\cap W(\A_k)_{\bullet}^{\Br}$, where $T$ is a finite set of places containing $S$, $U_v$ is an open subset of $W(k_v)$. 
Also there exists a point $(\alpha_v)_{v\in T}\times (\beta_v)_{v\notin T}\in \prod\limits_{v\in T}U_v\times \prod\limits_{v\notin T}\mathcal{W}(\ok)$ that is orthogonal to $\Br (W)$. 

By implicit function theorem, $f(U_v)\subset V(k_v)$ is an open subset, then  $(\prod\limits_{v\in T}f(U_v)\times \prod\limits_{v\notin T}\mathcal{V}(\ov))\cap V(\A_k)_{\bullet}^{\Br}$ is an open subset of $V(\A_k)_{\bullet}^{\Br}$ containing $(f(\alpha_v))_{v\in T}\times (f(\beta_v))_{v\notin T}$. 
Hence there exists an element $x\in V(k)$ contained in $\prod\limits_{v\in T}f(U_v)\times \prod\limits_{v\notin T}\mathcal{V}(\ov)$. 
The fiber $W_x$ satisfies $\overline{W_x(k)}=W_x(\A_k)_{\bullet}$. Moreover $x$ is an $\O_v$-point of $\mathcal{V}$ for all $v\notin T$, we can define $\mathcal{W}_{x,v}$ as the fiber of $\mathcal{W}$ over this $\O_v$-point. By construction we know that $\prod\limits_{v\in T}(U_v\cap W_x(k_v))\times \prod\limits_{v\notin T}\mathcal{W}_{x,v}(\ov)$ is a nonempty open subset of $W_{x}(\A_k)_{\bullet}$.
Therefore there exists an element $x'\in W_x(k)$ contained in $\prod\limits_{v\in T}(U_v\cap W_x(k_v))\times \prod\limits_{v\notin T}\mathcal{W}_{x,v}(\ov)\subset \Delta$.
This completes the proof.
\item Similar as proof of (1), except that the Brauer group is omitted.
\end{enumerate}
\end{proof}

\begin{cor}\label{coro5.5}
Given integers $n\geq2$, and $m_i\geq2$ for each $1\leq i\leq n$.
\begin{enumerate}
\item $W=\Spec \ k[T_{i,j}, 1\leq i\leq n,1\leq j\leq m_i]/(1-\sum\limits_{1\leq i\leq n}\prod\limits_{1\leq j\leq m_i}T_{i,j})$ satisfies $\overline{W(k)}=W(\A_k)_{\bullet}$,
\item $W'=\Spec \ k[T_{i,j}, 1\leq i\leq n,1\leq j\leq m_i; \frac{1}{\sum\limits_{1\leq i\leq n}\prod\limits_{1\leq j\leq m_i}T_{i,j}}]$ satisfies $\overline{W'(k)}=W'(\A_k)_{\bullet}^{\Br}$.
\end{enumerate}
\end{cor}
\begin{proof}
$(1)$ Consider the morphism  $f:W \to V=\Spec \ k[Z_{i,j}, 1\leq i\leq n,2\leq j\leq m_i]-\square$, which maps $Z_{i,j}$ to $T_{i,j}$, where $\square\subset \Spec \ k[Z_{i,j}, 1\leq i\leq n,2\leq j\leq m_i]$ is the closed subscheme defined by $(\prod\limits_{2\leq j\leq m_1}T_{1,j})=...=(\prod\limits_{2\leq j\leq m_n}T_{n,j})=0$.
Then $\overline{V(k)}=V(\A_k)_{\bullet}$ by Lemma \ref{codim 2}.
One can verify that \( f \) is a good fibration.  
Indeed, the fibre of \( f \) over any \( k \)-point has constant dimension; that is, all fibres are isomorphic to affine spaces of the same dimension.  
By miracle flatness (cf. \cite[Proposition~6.1.5]{ega4.2}), it follows that \( f \) is flat, and hence smooth.  
Moreover, the surjectivity on locally integral points arises from the geometrically integral fibres.

$(2)$ Set $D=\sum\limits_{1\leq i\leq n}\prod\limits_{1\leq j\leq m_i}T_{i,j}\in k[T_{i,j}]$, which is invertible on $W'$. Consider the morphism $W'\to W\times\Gm$ sending $(T_{i,1},\ldots,T_{i,m_i})_{1\leq i\leq n}$ to $\big((T_{i,1}/D,T_{i,2},\ldots,T_{i,m_i})_{1\leq i\leq n},\,D\big)$. This morphism is an isomorphism: its inverse sends $\big((T_{i,1},\ldots,T_{i,m_i})_{1\leq i\leq n},u\big)$ to $(uT_{i,1},T_{i,2},\ldots,T_{i,m_i})_{1\leq i\leq n}$. Hence $W'\cong W\times\Gm$.
Then $(1)$, Lemma \ref{product} and Lemma \ref{algebraic group} imply the result.
\end{proof}

Before proving Theorem \ref{th5.3}, we first introduce the following notation and establish two lemmas.

\begin{defn}
    For each $m\geq1$ and $n\geq 2$, let $\mathfrak{R}_{m,n}:=\{(r_\ell)_{1\leq \ell \leq n-1}\in \mathbb{Z}^{n-1}|1\leq r_\ell \leq m+\ell, r_1<r_2<...<r_{n-1} \}$, let $\mathfrak{S}_{m,n}:=\{(s_\ell)_{1\leq \ell\leq n}\in \mathbb{Z}^{n}|(s_\ell)_{1\leq \ell\leq n-1}\in\mathfrak{R}_{m,n}, s_{n-1}\leq s_n\leq m+n-2\}$.
\end{defn}

\begin{defn}
    For each integer $n\geq1$, and $r,s$ such that $1\leq r\leq s\leq n$, we define a set $\mathfrak{M}_n(s,r)$ as a subset of $k[T_{i,j}\mid_{1\leq i,j\leq n}]$ which contains all minors of order $r$ obtained by choosing any $r$ rows among the first $s$ rows and taking the first $r$ columns of the matrix $\left(T_{i,j}\right)_{n\times n}$.  
\end{defn}

\begin{defn}
    For positive integers $r,s$, such that $r\leq s$, we define the closed subscheme $\mathrm{SL}_{s,r}$ of $\mathrm{SL}_s$ (with coordinates $w_{i,j}$ where $1\leq i,j\leq s$) which is defined by $w_{s,i}=0$ for each $1\leq i\leq r-1$.
\end{defn}

\begin{lem}\label{sl}
    For positive integers $r,s$, such that $r< s$, we have $\overline{\mathrm{SL}_{s,r}(k)}=\mathrm{SL}_{s,r}(\A_k)_{\bullet}$.
\end{lem}

\begin{proof}
    For any matrix $(w_{i,j})_{1\leq i,j\leq s}$ in $\mathrm{SL}_{s,r}$, we have
    \[
    \det(w_{i,j})=\sum_{i=r}^s w_{s,i}\cdot C_{s,i},
    \]
    where $C_{s,i}$ is the cofactor of $(s,i)$-entry. Hence there always exists an integer $i_0$ with $r\leq i_0\leq s$ such that $C_{s,i_0}\neq 0$.
    Then there exists a natural morphism
    \[
        f:\mathrm{SL}_{s,r}\to \Spec \ k[w_{i,j}\mid 1\leq i\leq s-1,1\leq j\leq s]-\cap_{r\leq \ell\leq s}V(C_{s,\ell}).
    \]
    One checks that $f$ is a good fibration whose fibers are always isomorphic to $\Af^{s-r}$.
    Moreover, $r<s$ implies that $\cap_{r\leq \ell\leq s}V(C_{s,\ell})\subset \Spec \ k[w_{i,j}\mid 1\leq i\leq s-1,1\leq j\leq s]$ is a closed subscheme of codimension at least $2$.
    Therefore, we have $\overline{\mathrm{SL}_{s,r}(k)}=\mathrm{SL}_{s,r}(\A_k)_{\bullet}$ by Lemma \ref{codim 2} and Lemma \ref{lem5.4}.
\end{proof}

\begin{lem}\label{SA W}
    For each integer $m\geq1$ and $n\geq2$, for each $\mathbf{s}\in \mathfrak{S}_{m,n}$, let $W_{\mathbf{s}}$ be the closed subscheme of $\Spec \ k[T_{i,j}\mid_{1\leq i,j\leq m+n}]/(\det(T_{i,j})_{n\times n}-1)$ defined by
    $$\begin{cases}
T_{m+1+i,j}=0 &\text{ for any $1\leq i\leq n-2$ and $1\leq j\leq s_i$};\\
T_{m+n,j}=0 &\text{ for any $1\leq j\leq s_n$};\\    
f=0 &\text{ for any $1\leq i\leq n-1$ and any $f\in \mathfrak{M}_{m+n}(m+i-1,s_i)$}.
\end{cases}$$ 
Then $W_{\mathbf{s}}$ is a smooth and $k$-rational variety satisfying $\overline{W_{\mathbf{s}}(k)}=W_{\mathbf{s}}(\A_k)_{\bullet}$.

Similarly, let $W'_{\mathbf{s}}$ be the closed subscheme of $\Spec \ k[T_{i,j}\mid_{1\leq i,j\leq m+n};\frac{1}{\det(T_{i,j})}]$ defined by the same equations as above, then $W'_{\mathbf{s}}$ is a smooth and $k$-rational variety satisfying $\overline{W'_{\mathbf{s}}(k)}=W'_{\mathbf{s}}(\A_k)_{\bullet}^{\Br}$.
\end{lem}   

\begin{proof}
We only prove the case of $W_{\mathbf{s}}$. By a similar argument as in the proof of Corollary~\ref{coro5.5}(2), with $\det(T_{i,j})$ playing the role of $D$, $W'_{\mathbf{s}}$ is isomorphic to $W_{\mathbf{s}}\times\Gm$. The result then follows from the case of $W_{\mathbf{s}}$, together with Lemma \ref{product} and Lemma \ref{algebraic group}.

To prove the lemma, we fix $m$ and proceed by induction on $n$.

Given $m\geq1$, $n\geq2$ and $\textbf{s}=(s_\ell)_{1\leq \ell \leq n}\in \mathfrak{S}_{m,n}$ (set $s_0=0$ when $n=2$), note that in the coordinate ring of $W_{\textbf{s}}$, there is a decomposition (expanding the determinant along the ($m+n-1$)-th row) 
$$\det (T_{i,j})=\sum_{\ell=s_{n-2}+1}^{s_{n-1}} T_{m+n-1,\ell}\cdot C_{m+n-1,\ell},$$
where the cofactor $C_{m+n-1,\alpha}$ for $s_{n-2}+1\leq \alpha\leq s_{n-1}$ is of the form $$(\pm{1})\cdot\left|\begin{matrix}
T_{1,1} & ... & T_{1,\alpha-1} & T_{1,\alpha+1} &...& T_{1,s_n} & T_{1,s_n+1} &...&  T_{1,m+n}\\
\vdots &  & \vdots & \vdots & & \vdots & \vdots & &  \vdots\\
T_{m+n-2,1} & ... & T_{m+n-2,\alpha-1} & T_{m+n-2,\alpha+1} &...& T_{m+n-2,s_n} & T_{m+n-2,s_n+1} &...&  T_{m+n-2,m+n}\\
0 & ... & 0 & 0 &...& 0 & T_{m+n,s_n+1} &...&  T_{m+n,m+n}\\
\end{matrix}\right|.$$

Then $W_{\textbf{s}}$ has an open cover $W_{\textbf{s}}=\bigcup_\alpha D(C_{m+n-1,\alpha})$. On $D(C_{m+n-1,\alpha})$, column vectors $$\begin{pmatrix}
    T_{1,1}\\
    \vdots\\
    T_{m+n-2,1}
\end{pmatrix},..., \begin{pmatrix}
    T_{1,\alpha-1}\\
    \vdots\\
    T_{m+n-2,\alpha-1}
\end{pmatrix}, \begin{pmatrix}
    T_{1,\alpha+1}\\
    \vdots\\
    T_{m+n-2,\alpha+1}
\end{pmatrix},..., \begin{pmatrix}
    T_{1,s_{n-1}}\\
    \vdots\\
    T_{m+n-2,s_{n-1}}
\end{pmatrix}$$ are linearly independent.
Hence ``$f=0$ for any $f\in \mathfrak{M}_{m+n}(m+n-2,s_{n-1})$'' implies that $\begin{pmatrix}
    T_{1,\alpha}&\cdots&    T_{m+n-2,\alpha}
\end{pmatrix}^\mathrm{T}$ is a linear combination of the above vectors.

Therefore, when $n=2$, one checks that $D(C_{m+1,\alpha})$ can be embedded as a nonempty open subscheme of an affine space.
As $\bigcap_\alpha D(C_{m+1,\alpha})\neq\varnothing$, $W_{\textbf{s}}$ is a smooth and $k$-rational variety.

Similarly, when $n\geq3$, $D(C_{m+n-1,\alpha})$ can be embedded as an open subscheme of $ W'_{\textbf{s}'}\times\Af^{m+n-2+s_{n-1}-s_{n-2}}_k$, where $\textbf{s}'=(s_1,...,s_{n-2},s_n-1)\in\mathfrak{S}_{m,n-1}$ is a well-defined element.
As $\bigcap_\alpha D(C_{m+n-1,\alpha})\neq\varnothing$, the properties of smoothness and rationality of $W_{\textbf{s}}$ reduce to those of $W'_{\textbf{s}'}$.

In addition, we consider the morphism $$f:W_{\textbf{s}}\to \Spec \ k[X_{s_{n-2}+1},...,X_{s_{n-1}};Y_{s_{n-2}+1},...,Y_{s_{n-1}}]/(1-\sum_{i=s_{n-2}+1}^{s_{n-1}} X_i Y_i),$$
$$X_\alpha\mapsto T_{m+n-1,\alpha},$$
$$Y_\alpha\mapsto C_{m+n-1,\alpha}.$$
Note that for each $k$-point (similarly for other points) of the target variety, there exists an integer $s_{n-2}+1\leq\beta\leq s_{n-1}$ such that $Y_\beta\neq0$. On the coordinate ring of the fiber of $f$ at this point, column vectors $$\begin{pmatrix}
    T_{1,1}\\
    \vdots\\
    T_{m+n-2,1}
\end{pmatrix},..., \begin{pmatrix}
    T_{1,\beta-1}\\
    \vdots\\
    T_{m+n-2,\beta-1}
\end{pmatrix}, \begin{pmatrix}
    T_{1,\beta+1}\\
    \vdots\\
    T_{m+n-2,\beta+1}
\end{pmatrix},..., \begin{pmatrix}
    T_{1,s_{n-1}}\\
    \vdots\\
    T_{m+n-2,s_{n-1}}
\end{pmatrix}$$ are linearly independent.
Hence ``$f=0$ for any $f\in \mathfrak{M}_{m+n}(m+n-2,s_{n-1})$'' implies that $\begin{pmatrix}
    T_{1,\beta} &\cdots&T_{m+n-2,\beta}
\end{pmatrix}^\mathrm{T}$ is a linear combination of the above vectors. 
Therefore, when $n=2$, each fiber of $f$ is isomorphic to $\mathrm{SL}_{m+1,s_2-1}\times \Af^{m+2-s_1}$, hence $f$ is a good fibration, and $\overline{W_{\textbf{s}}(k)}=W_{\textbf{s}}(\A_k)_{\bullet}$ follows from Corollary \ref{coro5.5}.

Similarly, when $n\geq3$, each fiber of $f$ is isomorphic to $W_{\textbf{s}'}\times\Af^{m+n-s_{n-1}}$, where $\textbf{s}'=(s_1,...,s_{n-2},s_n-1)\in\mathfrak{S}_{m,n-1}$ is a well-defined element.
It then follows from the induction hypothesis that $f$ is a good fibration and $\overline{W_{\textbf{s}}(k)}=W_{\textbf{s}}(\A_k)_{\bullet}$.
This completes the proof.
\end{proof}

In fact, the proof of Lemma~\ref{SA W} closely resembles that of Theorem~\ref{th5.3}.

\

\begin{rem}\label{another def}
Recall that from linear algebra, we know that for a Jordan canonical form $M$ with a single eigenvalue $\lambda$, an invertible matrix $T$ satisfies  $T^{-1}MT$ is upper triangular iff $$\text{$(M-\lambda I)\cdot\text{``the  ($i+1$)-th column of $T$''}\in\text{Span\{``the first $i$ columns of $T$''\}}$ for each $i$},$$
i.e. $\{\text{``the first $i$ columns of $T$'', $(M-\lambda I)\cdot\text{``the  ($i+1$)-th column of $T$''}$}\}$ is linearly dependent. Or equivalently, the vanishing of certain minors related to $T$. 
This induces an alternative definition of $X_M$ as a closed subscheme of $\Spec \ k[T_{i,j},w]/(w\cdot\det(T)-1)$. We will use this definition in the subsequent discussion.

For example if $M=\begin{pmatrix}
            \lambda & 0 &0\\
            0 & \lambda & 1\\
            0 & 0 & \lambda
        \end{pmatrix}$, let $V$ be the closed subscheme of $\Spec \ k[T_{i,j},w]/(w\cdot\det(T)-1)$ defined by $T_{3,1}=0, T_{1,1}\cdot T_{3,2}=0$, then $V$ and $X_M$ have the same set of $\overline{k}$-points, hence they have isomorphic irreducible components. 
\end{rem}

By Remark \ref{another def}, we may define the variety \( X_M \) using an alternative set of equations. In particular, this new presentation admits further reduction, which simplifies subsequent computations.

\begin{lem}\label{lem:newdef}
Let \( X_M' \) be the closed subscheme of \(\Spec\ k[T_{i,j} \mid_{1 \leq i,j \leq m+n} ; w]/(w \cdot \det(T_{i,j}) - 1)\) defined by the equations:
\[
\begin{cases}
T_{i,1} = 0 & \text{for } m+2 \leq i \leq m+n;\\

f \cdot T_{i,j} = 0 & \text{for all } m+2\leq i\leq m+n,\ 2\leq j\leq i-1, \ f \in \mathfrak{M}_{m+n}(i-2, j-1).\\
\end{cases}
\]

Then \( X_M' \) and \( X_M \) have the same \(\overline{k}\)-points and thus the same reduced subscheme structure. In particular, they have isomorphic irreducible components.
\end{lem}

The proof of Lemma \ref{lem:newdef} is somewhat technical and is deferred to Appendix~\ref{newdef proof} for the sake of readability.

\

As \( X_M \) and \( X_M' \) have the same irreducible components, we shall henceforth denote \( X_M \) by the scheme \( X_M' \) for the rest of this section.

\begin{lem}\label{V_r}
    For each $\textbf{r}=(r_\ell)_{1\leq \ell \leq n-1}\in \mathfrak{R}_{m,n}$, define $V_{\textbf{r}}$ to be the closed subscheme of $\Spec \ k[T_{i,j}\mid_{1\leq i,j\leq m+n};w]/(w\cdot\det(T_{i,j})-1)$ defined by 
$$\begin{cases}
T_{m+1+i,j}=0 &\text{ for any $1\leq i\leq n-1$ and $1\leq j\leq r_i$};\\
f=0 &\text{ for any $1\leq i\leq n-1$ and any $f\in \mathfrak{M}_{m+n}(m+i-1,r_i)$}.
\end{cases}$$ 
Then each $V_{\textbf{r}}$ is a closed subscheme of $X_M$.
Moreover, we have
\[
X_M=\bigcup_{\textbf{r}\in \mathfrak{R}_{m,n}} V_{\textbf{r}}.
\] 
\end{lem}

\begin{proof}
    One can easily check that each $V_{\textbf{r}}$ is a closed subscheme of $X_M$ by comparing the defining equations of $V_{\textbf{r}}$ and $X_M$.

    Then let \( A = (a_{i,j}) \in X_M(\overline{k}) \).  
For any \( 1 \leq s, t \leq m+n \), let \( A^{s,t} \in \mathrm{M}_{s \times t}(\overline{k}) \) denote the submatrix of \( A \) consisting of its first \( s \) rows and first \( t \) columns.

For each \( 1 \leq i \leq n-1 \), define
\[
    r_i := \min\left\{ t \mid \mathrm{rank}(A^{m+i-1, t}) < t \right\}.
\]

We always have \( r_i \leq m+i \), since \( \mathrm{rank}(A^{m+i-1, m+i}) < m+i \).  
Moreover, one can easily verify that \( r_1 \leq r_2 \leq \cdots \leq r_{n-1} \), because \( \mathrm{rank}(A^{m+i-1, t}) \leq \mathrm{rank}(A^{m+i, t}) \).

However, we can in fact prove that the inequality is strict: \( r_1 < r_2 < \cdots < r_{n-1} \).  
Suppose, for the sake of contradiction, that there exists an integer \( 1 \leq i_0 \leq n-2 \) such that \( r_{i_0} = r_{i_0+1} \).

Then, we have \( a_{i,j} = 0 \) for all \( 1 \leq j \leq r_{i_0} \), \( m+i_0+1 \leq i \leq m+n \):
\begin{itemize}
    \item If $r_{i_0}=1$, then it follows directly from the defining equations of $X_M$.
    \item If $r_{i_0}\geq 2$, note that \( a_{i,1} = 0 \) for all \( m+i_0+1 \leq i \leq m+n \), and that\(f(A) \cdot a_{i,j} = 0\)for any \( 2 \leq j \leq r_{i_0} \), \( m+i_0+1 \leq i \leq m+n \), and any \( f \in \mathfrak{M}_{m+n}(i-2, j-1) \). By the definition of \( r_{i_0} \), for any \( 2 \leq j \leq r_{i_0} \), \( m+i_0+1 \leq i \leq m+n \), and any \( f \in \mathfrak{M}_{m+n}(i-2, j-1) \), there always exists $f\in \mathfrak{M}_{m+n}(i-2, j-1)$ such that $f(A)\neq 0$. Hence \( a_{i,j} = 0 \) for all \( 1 \leq j \leq r_{i_0} \), \( m+i_0+1 \leq i \leq m+n \).
\end{itemize}

Therefore, we have
\[
    \mathrm{rank}(A^{m+n, r_{i_0}}) = \mathrm{rank}(A^{m+i_0, r_{i_0}}) = \mathrm{rank}(A^{m+i_0, r_{i_0+1}}) < r_{i_0+1} =r_{i_0}.
\]
Thus, \( \mathrm{rank}(A) < m+n  \), which contradicts the assumption that \( A \) is invertible.

Hence, we conclude that \( r_1 < r_2 < \cdots < r_{n-1} \), i.e., \( \mathbf{r} = (r_\ell)_{1 \leq \ell \leq n-1} \in \mathfrak{R}_{m,n} \).  
It then follows easily that \( A \in V_{\mathbf{r}}(\overline{k}) \).  
This completes the proof.
\end{proof}

\ 

\begin{proof}[Proof of Theorem \ref{th5.3}]
According to  Lemma \ref{lem:newdef} and Lemma \ref{V_r}, we only need to prove that for any $\mathbf{r}\in \mathfrak{R}_{m,n}$, $V_{\textbf{r}}$ is a smooth and $k$-rational variety satisfying $\overline{V_{\textbf{r}}(k)}=V_{\textbf{r}}(\A_k)_{\bullet}^{\Br}$.

Our proof proceeds by induction on $n$. 
Assume that $\mathbf{r}=(r_\ell)_{1\leq \ell \leq n}\in \mathfrak{R}_{m,n}$ (set $r_0=0$ when $n=2$).
In fact, when $r_{n-1}< m+n-1$, one can easily check that $V_{\textbf{r}}$ is just $W'_{\textbf{s}}$ in Lemma \ref{SA W} where $\textbf{s}=(r_1,...,r_{n-1},r_{n-1})\in \mathfrak{S}_{m,n}$ is a well-defined element, so $V_{\textbf{r}}$ is a smooth and $k$-rational variety satisfying $\overline{V_{\textbf{r}}(k)}=V_{\textbf{r}}(\A_k)_{\bullet}^{\Br}$.
Hence we only need to consider the case where $r_{n-1}= m+n-1$.

In this case, by expanding the determinant along the ($m+n$)-th row, we have
\[
    \det (T_{i,j})=T_{m+n,m+n}\cdot C_{m+n,m+n},
\]
where the cofactor $C_{m+n,m+n}$ is of the form $$\left|\begin{matrix}
T_{1,1} & ... & T_{1,m+n-1}\\
\vdots &  & \vdots \\
T_{m+n-1,1} & ... & T_{m+n-1,m+n-1}\\
\end{matrix}\right|.$$

So one can easily check that when $n=2$, we have $V_{\textbf{r}}\cong \mathrm{GL}_{m+1}\times \Gm\times \Af^{m+1}_k$, hence is a smooth and $k$-rational variety satisfying $\overline{V_{\textbf{r}}(k)}=V_{\textbf{r}}(\A_k)_{\bullet}^{\Br}$ by Lemma \ref{algebraic group}; when $n\geq 3$, we have $V_{\textbf{r}}\cong V_{\textbf{r}'}\times\Gm\times \Af^{m+n-1}_k$, where $\textbf{r}'=(r_1,r_2,...,r_{n-2})\in \mathfrak{R}_{m,n-1}$ is a well-defined element, hence is a smooth and $k$-rational variety satisfying $\overline{V_{\textbf{r}}(k)}=V_{\textbf{r}}(\A_k)_{\bullet}^{\Br}$ by the induction hypothesis and Lemma \ref{product}.
This completes the proof.
\end{proof}

\begin{rem}
Notably, the authors have also established that Conjecture~\ref{co1} and Conjecture~\ref{co2} hold for all $M\in\mathrm{M}_n(\mathcal{O}_k)$ with $n\leq 5$.
All the results so far are obtained using the fibration method. However, different fibrations are employed for different cases, making it difficult to unify them.
Consequently, the validity of these conjectures in full generality remains open.

Moreover, the causal relationship between Conjecture~\ref{co1} and Conjecture~\ref{co2} is still unclear, and it is also unknown whether, when $\mathcal{O}_K$ is a PID, Theorem~\ref{hasse} can be deduced directly from Theorem~\ref{bm} and Conjecture~\ref{co2}.
\end{rem}

\

{\bf Acknowledgments}	
{\it We sincerely thank our thesis advisor Professor Yongqi Liang for raising such an engaging question. We also thank Professor Yang Cao for his valuable discussions. We are grateful to Professor Fei Xu for kindly providing us with a relevant reference.}

\

\appendix

\section{An explicit Brauer--Manin obstruction for triangularizability}\label{explicit appendix}

Let $k=\Q(\sqrt{-5})$ be the number field with ring of integers $\ok=\Z[\sqrt{-5}]$. 
The ideal class group $\mathrm{Cl}(\ok)$ is cyclic of order $2$.
Let $\p=(2,1+\sqrt{-5})$ and $\mathfrak{q}=(3,1+\sqrt{-5})$.
Then $\p$ and $\mathfrak{q}$ represent the unique nontrivial ideal class with $(2)=\p^2$ and $(1+\sqrt{-5})=\p\mathfrak{q}$. 

Take
\[
M=\begin{pmatrix}
0 & 4\\
2\sqrt{-5}-4 & 0
\end{pmatrix}\in \mathrm{M}_2(\ok),
\]
with eigenvalues
\[
\lambda_+=2(1+\sqrt{-5}),\qquad \lambda_-=-2(1+\sqrt{-5}),
\]
and corresponding eigenvectors
\[
v_+=\begin{pmatrix}2\\1+\sqrt{-5}\end{pmatrix},\qquad v_-=\begin{pmatrix}-2\\1+\sqrt{-5}\end{pmatrix}.
\]
Thus $M$ is triangularizable over $k$, hence by Proposition \ref{tp}  it is triangularizable over $\op$ for every finite place $\p$ of $k$, i.e.\ $\prod_v\mathcal{X}_M(\ov)\neq\varnothing$.

The variety
\(
X_M=X_{M,\mathrm{red}}=X_+\sqcup X_-
\)
is the disjoint union of two irreducible components, where
\[
X_M=\Spec\,k[T_{ij},w]\Big/\Bigl(\bigl((1+\sqrt{-5})T_{11}-2T_{21}\bigr)\bigl((1+\sqrt{-5})T_{11}+2T_{21}\bigr),\ w\det T-1\Bigr),
\]
and
\[
X_+=\Spec\frac{k[T_{ij},w]}
{\bigl((1+\sqrt{-5})T_{11}-2T_{21},\ w\det T-1\bigr)},\ \ 
X_-=\Spec\frac{k[T_{ij},w]}
{\bigl((1+\sqrt{-5})T_{11}+2T_{21},\ w\det T-1\bigr)}.
\]
The integral model of \(X_M\) is given by
\[
\mathcal{X}_M
=
\Spec\,\ok[T_{ij},w]
\Big/\Bigl(\bigl((1+\sqrt{-5})T_{11}-2T_{21}\bigr)\bigl((1+\sqrt{-5})T_{11}+2T_{21}\bigr),\ w\det T-1\Bigr).
\]
Since \(X_M\) is the disjoint union of \(X_+\) and \(X_-\), it follows that
\(
\Br(X_M)\cong \Br(X_+)\oplus \Br(X_-).
\)

We consider the quaternion algebra $(-1,T_{21}/(1+\sqrt{-5}))$ over $k(X_+)$ and $k(X_-)$, which defines  classes in $\br(X_+)$ and $\br(X_-)$ respectively, since $T_{21}/(1+\sqrt{-5})$ is an invertible regular function on both $X_+$ and $X_-$.
Let
\[
\beta=\left(-1,\dfrac{T_{21}}{1+\sqrt{-5}}\right)\oplus\left(-1,\dfrac{T_{21}}{1+\sqrt{-5}}\right)\in\br(X_{M}),
\]
where the first summand is taken on $X_+$ and the second on $X_-$.
As $k$ is totally imaginary, the local invariant vanishes at the archimedean place. Let $(x_v)\in\prod_v\mathcal{X}_M(\ov)$ be any adelic integral point.

\begin{itemize}
\item \textbf{For $v\neq\p,\mathfrak{q}$:} We have $\ord_v(2)=\ord_v(1+\sqrt{-5})=0$. On $X_+$, the relation $(1+\sqrt{-5})T_{11}=2T_{21}$ together with $\det T\in\ov^\times$ forces $T_{21}/(1+\sqrt{-5})\in\ov^\times$; the same holds on $X_-$ from $(1+\sqrt{-5})T_{11}=-2T_{21}$. In either case $\mathrm{inv}_v(\beta(x_v))=0$, since the Hilbert symbol $(-1,u)_v=1$ for any unit $u\in\ov^\times$.

\item \textbf{For $v=\mathfrak{q}$:} We have $\ord_{\mathfrak{q}}(1+\sqrt{-5})=1$ and $\ord_{\mathfrak{q}}(2)=0$. On $X_+$, comparing valuations in $(1+\sqrt{-5})T_{11}=2T_{21}$ with $\det T\in\mathcal{O}_{\mathfrak{q}}^\times$ forces $\ord_{\mathfrak{q}}(T_{11})=0$ and $\ord_{\mathfrak{q}}(T_{21})=1$; the same conclusion holds on $X_-$ via $(1+\sqrt{-5})T_{11}=-2T_{21}$. In either case $T_{21}/(1+\sqrt{-5})\in\mathcal{O}_{\mathfrak{q}}^\times$, so $\mathrm{inv}_{\mathfrak{q}}(\beta(x_{\mathfrak{q}}))=0$.

\item \textbf{For $v=\p$:} Write $1+\sqrt{-5}=\pi u$ with $\pi$ a uniformizer and $u\in\mathcal{O}_\p^\times$. We have $\ord_\p(2)=2$ and $\ord_\p(1+\sqrt{-5})=1$. On $X_+$, comparing valuations in $(1+\sqrt{-5})T_{11}=2T_{21}$ with $\det T\in\mathcal{O}_\p^\times$ forces $\ord_\p(T_{21})=0$, so $T_{21}/(1+\sqrt{-5})=\pi^{-1}u'$ for some $u'\in\mathcal{O}_\p^\times$; the same holds on $X_-$ via $(1+\sqrt{-5})T_{11}=-2T_{21}$. Since $k_\p=\Q_2(\sqrt3)$, the extension $k_\p(\sqrt{-1})/k_\p$ is unramified, so for $a\in k_\p^\times$ the Hilbert symbol $(-1,a)_\p$ equals $1$ if $\ord_\p(a)$ is even and $-1$ if $\ord_\p(a)$ is odd. As $\ord_\p(\pi^{-1}u')=-1$ is odd, we get $(-1,\pi^{-1}u')_\p=-1$, so in either case $\mathrm{inv}_\p(\beta(x_\p))=1/2$.
\end{itemize}

Hence $\sum_v\mathrm{inv}_v(\beta(x_v))=1/2\neq0$ for every $(x_v)\in\prod_v\mathcal{X}_M(\ov)$.
Then we conclude that $X_{M}(\A_k)^{\br}_\bullet\cap\prod_v\mathcal{X}_M(\ov)=\varnothing$, i.e. there is a Brauer--Manin obstruction to the existence of integral points on $X_M$.

\

\section{Proof of Lemma~\ref{lem:newdef}}\label{newdef proof}
In this appendix, we provide the detailed proof of Lemma~\ref{lem:newdef}.

\begin{proof}[Proof of Lemma \ref{lem:newdef}]
    By permutating the indices $i$ and $j$, we first rewrite the defining equations of $X_M'$ as the following system:
    \[
\begin{cases}
T_{i,1}=0 & \text{for } m+2\leq i\leq m+n;\\
f\cdot T_{i,j}=0 & \text{for all } 2\leq j\leq m+n-1,\ \mathrm{max}(m+2,j+1)\leq i\leq m+n,f\in\mathfrak{M}_{m+n}(i-2,j-1).
\end{cases}
\]

According to Remark \ref{another def}, $X_M$ has the same $\overline{k}$-points with the closed subschemes of \(\Spec\ k[T_{i,j} \mid_{1 \leq i,j \leq m+n} ; w]/(w \cdot \det(T_{i,j}) - 1)\) defined by all full-column minors of the matrices below, where \(1 \leq \ell \leq m+n\):
\[
    M_\ell=\begin{pmatrix}
        T_{1,1} & \cdots & T_{1,\ell-1} & 0\\
        \vdots &  & \vdots & \vdots\\
        T_{m,1} & \cdots & T_{m,\ell-1} & 0\\
        T_{m+1,1} & \cdots & T_{m+1,\ell-1} & T_{m+2,\ell}\\
        \vdots &  & \vdots & \vdots\\
        T_{m+n-1,1} & \cdots & T_{m+n-1,\ell-1} & T_{m+n,\ell}\\ 
        T_{m+n,1} & \cdots & T_{m+n,\ell-1} & 0
    \end{pmatrix}.
\]

We only need to show that the ideal generated by all full-column minors of the matrices \(M_1,\ldots,M_{m+n}\) is equal to the ideal generated by the defining equations of \(X_M'\).

For each \(1\leq \ell \leq m+n\), let \(\mathfrak{I}_\ell\) denote the ideal in \(k[T_{i,j} \mid_{1\leq i,j\leq m+n} ]\) generated by all full-column minors of the matrices \(M_1,\ldots,M_\ell\). Clearly, we have \(\mathfrak{I}_1=(T_{m+2,1},\ldots,T_{m+n,1})\).

We claim that for each \(2\leq \ell \leq m+n-1\), the ideal \(\mathfrak{I}_\ell\) is generated by the union of \(\mathfrak{I}_{\ell-1}\) and the set
\[
\mathcal{M}_\ell = \{f\cdot T_{i,\ell} \mid \mathrm{max}(m+2,\ell+1)\leq i\leq m+n,\ f\in\mathfrak{M}_{m+n}(i-2,\ell-1)\}.
\]

Now we prove the claim.
Fix such an \(\ell\), and consider the full-column minors of \(M_\ell\).

For integers \(a\geq b\geq1\), define
\[
\Upsilon_{a,b}:=\{(u_i)_{1\leq i\leq b}\in \Z^b \mid 1\leq u_1<u_2<\cdots<u_b\leq a\}.
\]

For \(2\leq e\leq m+n\), let \(M_{\mathbf{u}}\) be the submatrix of \(M_e\) obtained by selecting all columns and the rows indexed by \(\mathbf{u}=(u_i)_{1\leq i\leq e}\in \Upsilon_{m+n,e}\), so its determinant is a full-column minor of \(M_e\).

For \(\mathbf{u}=(u_i)\in \Upsilon_{m+n,\ell}\), we have
\[
\det(M_{\mathbf{u}}) =
\begin{cases}
0, & \text{if } u_\ell \leq m;\\
T_{u_\ell+1,\ell}\cdot \mathcal{C}_{\ell,\ell} + \sum_{i=1}^{\ell-1} T_{u_\ell,i}\cdot \mathcal{C}_{\ell,i}, & \text{if } m+1\leq u_\ell < m+n;\\
\sum_{i=1}^{\ell-1} T_{u_\ell,i}\cdot \mathcal{C}_{\ell,i}, & \text{if } u_\ell = m+n,
\end{cases}
\]
where \(\mathcal{C}_{\ell,i}\) denotes the cofactor of the \((\ell,i)\)-entry in \(M_{\mathbf{u}}\).

When \(m+1\leq u_\ell < m+n\), we always have \(\mathcal{C}_{\ell,\ell}\in \mathfrak{M}_{m+n}(u_\ell-1,\ell-1)\), so the term \(T_{u_\ell+1,\ell}\cdot \mathcal{C}_{\ell,\ell}\) lies in the ideal generated by \(\mathcal{M}_\ell\). 
Also, for any $\mathrm{max}(m+2,\ell+1)\leq \ell' \leq m+n$ and any $f\in \mathfrak{M}_{m+n}(\ell'-2,\ell-1)$, the term $f\cdot T_{\ell',\ell}$ appears as the term \(T_{u_\ell,\ell}\cdot \mathcal{C}_{\ell,\ell}\) when we take some \(\mathbf{u}\in \Upsilon_{m+n,\ell}\) with \(u_\ell = \ell'\).
Thus, it suffices to show that each term \(T_{u_\ell,i}\cdot \mathcal{C}_{\ell,i} \in \mathfrak{I}_{\ell-1}\) for \(1\leq i\leq \ell-1\) when \(u_\ell \geq m+1\).
When $i = 1$, we have $T_{u_\ell,1} \cdot \mathcal{C}_{\ell,1} \in \mathfrak{I}_{\ell-1}$, since either $T_{u_\ell,1} \in \mathfrak{I}_1 \subset \mathfrak{I}_{\ell-1}$ when $u_\ell \geq m+2$, or $\mathcal{C}_{\ell,1} = 0$ when $u_\ell = m+1$.

Therefore, we only need to consider those $\mathbf{u}=(u_i)\in \Upsilon_{m+n,\ell}$ with \(u_\ell \geq m+1\) and show that for each \(2\leq i\leq \ell-1\), the term \(T_{u_\ell,i}\cdot \mathcal{C}_{\ell,i} \in \mathfrak{I}_{\ell-1}\).

If \(u_{\ell-1} < u_\ell - 1\), then \(\mathcal{C}_{\ell,i}\) is generated by \(\mathfrak{M}_{m+n}(u_\ell-1,i-1)\), so \(T_{u_\ell,i}\cdot \mathcal{C}_{\ell,i} \in \mathfrak{I}_{\ell-1}\).
Thus we only need to consider the case \(u_{\ell-1} = u_\ell - 1\).

If \(u_{\ell-1} = u_\ell - 1\), then
\[
T_{u_\ell,i}\cdot \mathcal{C}_{\ell,i} = T_{u_\ell,i} \cdot \sum_{\substack{1\leq j\leq \ell\\ j\neq i}} (-1)^{\delta_{i,j}}\cdot T_{u_{\ell-1},j} \cdot \det(M_{\mathbf{u}\mid_{i,j}}),
\]
where \(\delta_{i,j} \in \{0,1\}\), a sign factor that is irrelevant in our argument, and \(M_{\mathbf{u}\mid_{i,j}}\) is the submatrix of \(M_{\mathbf{u}}\) obtained by deleting the \((\ell-1)\)-th and \(\ell\)-th rows and the \(i\)-th and \(j\)-th columns.

If \(j > i\), then \(\det(M_{\mathbf{u}\mid_{i,j}})\) is generated by \(\mathfrak{M}_{m+n}(u_\ell-2,i-1)\), so \(T_{u_\ell,i} T_{u_{\ell-1},j}\cdot \det (M_{\mathbf{u}\mid_{i,j}}) \in \mathfrak{I}_i\subset \mathfrak{I}_{\ell-1}\). 
Hence, such expansion terms are eliminated, and we are left with only those terms with \(j < i\).

We observe that the expansion proceeds row by row, starting from the last row and moving upward to the second-to-last. 
In this process, we eliminate any expansion terms whose second-to-last-row entry lies to the right of the last-row entry. 
Continuing in this manner, we stop once the index of the row being expanded, as measured in $M_\ell$, is less than $m+2$. 
At that point, if there exists some $\ell'$ such that $u_{\ell'} > m+1$ and $u_{\ell'} > u_{\ell'-1} + 1$, then all remaining expansion terms will lie in $\mathfrak{I}_{\ell-1}$, and hence be eliminated. Otherwise, the column indices of all remaining terms get smaller as we go up.

Then we have two cases to consider:
\begin{itemize}
    \item When \( u_2 > m+1 \), if there is any surviving expansion term, it must involve a full expansion across all \( \ell \) rows. In this case, the rows selected by \( \mathbf{u} \) are pairwise adjacent, and the column indices decrease strictly from bottom to top. But this is not possible: although there are \( \ell \) rows, our initial column indices are at most \( \ell - 1 \). In other words, there are not enough distinct columns to support a strictly decreasing sequence. So no expansion terms can survive. 
    \item When \( u_2 \leq m+1 \), then there exists some \( \ell' > 1 \) such that \( u_{\ell'} \leq m+1 \) and \( u_{\ell'+1} > m+1 \). If there is any surviving expansion term, the expansion process must stop at the \( u_{\ell'} \)-row, with selected entries having strictly decreasing column indices from bottom to top. The submatrix of \( M_{\mathbf{u}} \) obtained by removing the rows and columns corresponding to these entries has determinant zero, since its rightmost column is entirely zero. Therefore, all remaining expansion terms  are eliminated.
\end{itemize}

Thus, we have completed the proof of the claim that \(\mathfrak{I}_\ell\) is generated by \(\mathfrak{I}_{\ell-1}\) and \(\mathcal{M}_\ell\) for each \(2\leq \ell \leq m+n-1\).
A similar argument shows that \(\mathfrak{I}_{m+n} = \mathfrak{I}_{m+n-1}\), which completes the proof of the lemma.
\end{proof}

\

\bibliographystyle{amsalpha}
\bibliography{mybib}

 \end{document}